\documentclass[12pt,twoside]{amsart}

\usepackage[psamsfonts]{amssymb}
\usepackage{times,a4wide}
\usepackage[draft=false]{hyperref}

\theoremstyle{plain}
\newtheorem{thm}{Theorem}
\newtheorem{prop}[thm]{Proposition}
\newtheorem{cor}[thm]{Corollary}
\newtheorem{lemma}[thm]{Lemma}
\newtheorem{claim}[thm]{Claim}
\theoremstyle{remark}
\newtheorem*{rem}{Remark}
\newtheorem*{rems}{Remarks}
\theoremstyle{definition}
\newtheorem{defn}{Definition}

\newcommand{\C}{{\mathbb C}}
\newcommand{\R}{{\mathbb R}}
\newcommand{\D}{{\mathbb D}}

\newcommand{\K}{\mathcal{K}}
\newcommand{\iC}{{\int_{\C}}}
\newcommand{\E}{{\mathbb E}}
\newcommand{\V}{{\mathbb V}}
\newcommand{\Pro}{{\mathbb P}}
\newcommand{\lil}{{\lambda\in\Lambda}}
\newcommand{\lpil}{{\lambda'\in\Lambda}}

\newcommand{\Fdosa}{{\mathcal F_{\phi}^2}}

\newcommand{\FL}{{\mathcal F_L^2}}

\newcommand\renorm[1]{\ensuremath{|\!|\!|#1|\!|\!|}}
\newcommand\inprod[2]{\ensuremath{\langle\!\langle#1,#2\rangle\!\rangle}}
\newcommand{\supp}{\operatorname{supp}}
\newcommand{\re}{\operatorname{Re}}

\newenvironment{alist}
{

\begin{enumerate}}
{\end{enumerate}}

\title{Inhomogenous random zero sets.}
\thanks{The authors are supported by the Generalitat de Catalunya (grant 2009 SGR
1303) and the Spanish Ministerio de Econom\'ia y Competividad (project MTM2011-27932-C02-01)}
\author{Jeremiah Buckley}\address{Dept.\ Matem\`atica Aplicada i An\`alisi, Universitat  de Barcelona, Gran Via 585, 08007 Bar\-ce\-lo\-na, Spain}\email{jerry.buckley@ub.edu}
\author{Xavier Massaneda}\email{xavier.massaneda@ub.edu}
\author{Joaquim Ortega-Cerd\`a}\email{jortega@ub.edu}

\begin{document}
\begin{abstract}
We construct random point processes in $\C$ that are asymptotically close to a given doubling measure. The processes we construct are the zero sets of
random entire functions that are constructed through generalised Fock spaces. We offer two alternative constructions, one via bases for these spaces and
another via frames, and we show that for both constructions the average distribution of the zero set is close to the given doubling measure, and that the
variance is much less than the variance of the corresponding Poisson point process. We prove some asymptotic large deviation estimates for these
processes, which in particular allow us to estimate the `hole probability', the probability that there are no zeroes in a given open bounded subset of the
plane. We also show that the `smooth linear statistics' are asymptotically normal, under an additional regularity hypothesis on the measure. These
generalise previous results by Sodin and Tsirelson for the Lebesgue measure.
\end{abstract}
\maketitle

\section{Introduction}

In this paper we are interested in random point processes that mimic a given $\sigma$-finite measure $\mu$ on the complex plane. The classical example is
the inhomogenous Poisson point process, which we consider in the following manner. Fix a parameter $L>0$ and let $N_L$ be the Poisson random measure on
$\C$ with intensity $L\mu$, that is,
\begin{itemize}
\item $N_L$ is a random measure on $\C$,
\item For every measurable $A\subset\C$, $N_L(A)$ is a Poisson random variable with mean $L\mu(A)$, and
\item If $A$ and $B$ are disjoint then $N_L(A)$ and $N_L(B)$ are independent.
\end{itemize}
Such an $N_L$ always exists, see for example \cite[Proposition 19.4]{Sat}. Suppose that $\psi\in L^1(\mu)\cap L^2(\mu)$ and define
\[
N(\psi,L)=\frac1L\iC\psi(z)dN_L(z).
\]
Then (see \cite[Proposition 19.5]{Sat})
\begin{equation}\label{poismean}
\E\left[N(\psi,L)\right]=\iC\psi d\mu
\end{equation}
and, writing $\V$ for the variance,
\begin{equation}\label{poisvar}
\V[N(\psi,L)]=\frac1L\iC|\psi|^2d\mu.
\end{equation}

In contrast to the Poisson point process, the zero sets of random analytic functions are known to be more `rigid' processes, in particular these processes
exhibit `local repulsion' (see \cite[Chapter 1]{HKPV}). We will construct a random zero set such that $\eqref{poismean}$ continues to hold (at least for
smooth $\psi$ in the limit $L\rightarrow\infty$, see Theorem~\ref{almostsurebasis}) but with a variance that decays faster than $L^{-2}$, in contrast to
$\eqref{poisvar}$ (Theorem~\ref{variance}). In fact we will also have
\[
N(\psi,L)\rightarrow\iC\psi d\mu\text{ as }L\rightarrow\infty
\]
almost surely, as well as being true in mean (Theorem~\ref{almostsurebasis}).

As a further measure of the `rigidity' of our process we note that the `hole probability' for the Poisson point process is, by definition,
\[
\Pro[N_L(A)=0]=e^{-L\mu(A)}
\]
for any $A\subset\C$ whereas we shall see that the `hole probability' for the zero sets we construct decays at least like $e^{-cL^2}$ for some $c>0$.

This problem has already been considered when $\mu$ is the Lebesgue measure on the plane (\cite{ST1} and \cite{ST3}) and the resultant zero-sets are
invariant in distribution under plane isometries. We are interested in generalising this construction to other measures, where we cannot expect any such
invariance to hold. We begin by recalling the following definition.
\begin{defn}\label{doubling}
A nonnegative Borel measure $\mu$ in $\C$ is called \textit{doubling} if  there exists $C>0$ such that
\[\mu(D(z,2r))\leq C \mu(D(z,r))\]
for all $z\in\C$ and $r>0$. We denote by $C_\mu$ the infimum of the constants $C$ for which the inequality holds, which is called the \textit{doubling
constant for $\mu$}.
\end{defn}
Let $\mu$ be a doubling measure and let $\phi$ be a subharmonic function with $\mu=\Delta\phi$. Canonical examples of such functions are given by
$\phi(z)=|z|^\alpha$ where $\alpha>0$ (the value $\alpha=2$ corresponds of course to the Lebesgue measure). The function $\phi(z)=(\re z)^2$ gives a
non-radial example, and more generally one can take $\phi$ to be any subharmonic, non-harmonic, (possibly non-radial) polynomial. We define, for
$z\in\mathbb{C}$, $\rho_\mu(z)$ to be the radius such that $\mu(D(z,\rho_\mu(z)))=1$. We shall normally ignore the dependence on $\mu$ and simply write
$\rho(z)$.

Consider the generalised Fock space
\[
\Fdosa =\{f\in H(\mathbb{C}):\|f\|_\Fdosa^2=\iC|f(z)|^2e^{-2\phi(z)}\frac{dm(z)}{\rho(z)^2}<+\infty\}
\]
where $m$ is the Lebesgue measure on the plane. We note that, as in \cite{Ch}, the measure $\frac{dm(z)}{\rho(z)^2}$ can be thought of as a regularisation
of the measure $d\mu(z)$. The classical Bargmann-Fock space corresponds to $\phi(z)=|z|^2$. Let $(e_n)_n$ be an orthonormal basis for the space $\Fdosa$
and $(a_n)_n$ be a sequence of independent standard complex Gaussian random variables (that is, the probability density of each $a_n$ is
$\frac{1}{\pi}\exp(-|z|^2)$ with respect to the Lebesgue measure on the plane; we denote this distribution $\mathcal{N}_\C(0,1)$). Consider the Gaussian
analytic function (GAF) defined by
\[
g(z)=\sum_n a_n e_n (z).
\]
This sum almost surely defines an entire function (see for example \cite[Lemma~2.2.3]{HKPV}). The covariance kernel associated to this function is given
by (note that $\E[g(z)]=0$ for all $z\in\C$)
\[
\K(z,w)=\E[g(z)\overline{g(w)}]=\sum_ne_n (z)\overline{e_n (w)}
\]
which is the reproducing kernel for the space $\Fdosa$. Moreover the distribution of the random analytic function $g$ is determined by the kernel $\K$ so
it does not matter which basis we chose.

We are interested in studying the zero set $\mathcal Z(g)$, and a first observation is that since $g(z)$ is a mean-zero, normal random variable with
variance $\K(z,z)\neq0$ (see Proposition~\ref{kernest}), g has no deterministic zeroes. Furthermore the random zeroes of $g$ are almost surely simple
(\cite[Lemma~2.4.1]{HKPV}). We study the zero set $\mathcal Z(g)$ through the counting measure
\[
n_g=\frac1{2\pi}\Delta\log|g|
\]
(this equality is to be understood in the distributional sense). The Edelman-Kostlan formula (\cite[Theorem 1]{Sod} or \cite[Section 2.4]{HKPV}) for the
density of zeroes gives
\[\E[n_g(z)]=\frac{1}{4\pi}\Delta\log \K(z,z) dm(z).\]
We finally note that $\Delta\log \K(z,z)\simeq\frac{1}{\rho(z)^2}$ (see Section~\ref{kernels}) which, as we have already noted, can be viewed as a
regularisation of the measure $\mu$.

We will modify this construction by re-scaling the weight $\phi$, so that the zeroes will be even better distributed. Specifically, let $L$ be a positive
parameter and consider instead the weight $\phi_L=L\phi$ (and $\rho_L=\rho_{L\mu}$). For each $L$ we take a basis $(e_n^L)_n$ for the space $\FL=\mathcal
F_{\phi_L}^2$ and define
\begin{equation}\label{basisdefn}
g_L(z)=\sum_n\ a_n e_n^L (z)
\end{equation}
and
\[
\K_L(z,w)=\E[g(z)\overline{g(w)}]=\sum_ne_n^L (z)\overline{e_n^L (w)}.
\]

The following result states that the corresponding zero set, suitably scaled, is well distributed with respect to the measure $\mu$ for large values of
$L$.
\begin{thm}\label{almostsurebasis}
Let $\psi$ be a smooth real-valued function with compact support in $\C$ (which we always assume is not identically zero), let $n_L$ be the counting
measure on the zero set of $g_L$ and define the random variable $n(\psi,L)=\frac1L\int\psi dn_L$.

\noindent (a)
\[
\left|\E\left[n(\psi,L)\right]-\frac{1}{2\pi}\int\psi d\mu\right|\lesssim\frac{1}{L}\iC|\Delta\psi(z)|dm(z),
\]
where the implicit constant depends only on the doubling constant of the measure $\mu$.

\noindent (b) If we restrict $L$ to taking integer values then, almost surely,
\[
n(\psi,L)\rightarrow\frac{1}{2\pi}\int\psi d\mu
\]
as $L\rightarrow\infty$.
\end{thm}
The proof of part (b) of this result uses an estimate on the the decay of the variance of $n(\psi,L)$ which is interesting by itself.
\begin{thm}\label{variance}
For any smooth function $\psi$ with compact support in $\C$
\[
\V[n(\psi,L)]\simeq\frac1{L^2}\iC\Delta(\psi(z))^2\rho_L(z)^2dm(z).
\]
\end{thm}
\begin{rem}
We may estimate the dependence on $L$ using \eqref{rhoL} to see that the integral decays polynomially in $L$. If the measure $\mu$ is locally flat (see
Definition~\ref{localflat}) then we see that the variance decays as $L^{-3}$, just as in \cite{ST1}.
\end{rem}

In the special case $\phi(z)=|z|^2/2$ (the factor $1/2$ is simply a convenient normalisation) it is easy to see that the set
$(\frac1{\pi\sqrt{2}}\frac{(\sqrt{L}z)^n}{\sqrt{n!}})_{n=0}^\infty$ is an orthonormal basis for the corresponding Fock space, so that the construction
just given corresponds to the GAF studied in \cite{ST1} and \cite{ST3}. More generally if $\phi(z)=|z|^\alpha/2$ and $\alpha>0$ then the set
$(\frac{(L^{1/\alpha}z)^n}{c_{\alpha n}})_{n=0}^\infty$ is an orthonormal basis for the corresponding Fock space, for some $c_{\alpha
n}\simeq\Gamma(\frac2\alpha n+1)^{1/2}$ (and the implicit constants depend on $\alpha$, see the Appendix).

However, besides these special cases, we have very little information about the behaviour of an orthonormal basis for $\FL$. For this reason we also study
random functions that are constructed via frames.
\begin{defn}
Let $(X,\langle\cdot,\cdot\rangle)$ be an inner product space. A sequence $(x_n)_n$ in $X$ is said to be a \textit{frame} if there exist $0<A\leq B$ such
that
\[
A\|x\|^2\leq\sum_n|\langle x,x_n\rangle|^2\leq B\|x\|^2
\]
for all $x\in X$.
\end{defn}
It can be shown that this implies that there exists a sequence $(\tilde{x}_n)_n$ in $X$ (the canonical dual frame) such that
\[
x=\sum_n\langle x,\tilde{x}_n\rangle x_n
\]
and
\[
\frac1B\|x\|^2\leq\sum_n|\langle x,\tilde{x}_n\rangle|^2\leq\frac1A\|x\|^2
\]
for all $x\in X$. Thus a frame can be thought of as a generalisation of a basis that retains the spanning properties of a basis although the elements of
the frame are not, in general, linearly independent. (For a proof of the above facts and a general introduction to frames see, for example, \cite[Chapter
5]{Chr}.)

We will consider frames for $\FL$ consisting of normalised reproducing kernels, $k_\zeta (z)=\frac{\K_L(z,\zeta)}{\K_L(\zeta,\zeta)^{1/2}}$ (we ignore the
dependence on $L$ to simplify the notation). We consider frames of the form $(k_\lambda)_{\lil_L}$ (where the index set $\Lambda_L\subset\C$ is a sampling
sequence, see Section~\ref{kernels} for the definition). The advantage of this approach is that we have estimates for the size of the reproducing kernel
(Theorem~\ref{kernest}), and so we also have estimates for the size of the frame elements. We now define
\begin{equation}\label{framedefn}
f_L(z)=\sum_{\lil_L} a_\lambda k_\lambda (z)
\end{equation}
where $a_\lambda$ is a sequence of iid $\mathcal{N}_\C(0,1)$ random variables indexed by the sequence $\Lambda_L$. The covariance kernel for $f_L$ is
given by
\[
K_L(z,w)=\E[f_L(z)\overline{f_L(w)}]=\sum_{\lil_L} k_\lambda (z)\overline{k_\lambda (w)}
\]
which satisfies similar estimates to $\K_L$ (see Proposition~\ref{covkerest}).

Since the proof of Theorem~\ref{almostsurebasis} uses only estimates for the size of the covariance kernel we may state an identical theorem for the GAF
defined via frames. However in this case we also have the following stronger result.

\begin{thm}\label{largedevsmooth}
Let $n_L$ be the counting measure on the zero set of the GAF $f_L$ defined via frames \eqref{framedefn}, $\psi$ be a smooth real-valued function with
compact support in $\C$, $n(\psi,L)=\frac1L\int\psi dn_L$.

\noindent (a)
\[
\left|\E\left[n(\psi,L)\right]-\frac{1}{2\pi}\int\psi d\mu\right|\lesssim\frac{1}{L}\iC|\Delta\psi(z)|dm(z),
\]
where the implicit constant depends only on the doubling constant of the measure $\mu$.

\noindent (b) Let $\delta>0$. There exists $c>0$ depending only on $\delta$, $\psi$ and $\mu$ such that
\begin{equation}\label{largedevsmootheq}
\Pro\left[\Big|\frac{n(\psi,L)}{\frac{1}{2\pi}\int\psi d\mu}-1\Big|>\delta\right]\leq e^{-cL^2}.
\end{equation}
as $L\rightarrow\infty$.
\end{thm}
The proof of part (a) is identical to the proof of Theorem~\ref{almostsurebasis} (a) (using the appropriate estimates for the covariance kernel of $f_L$).
It is also easy to see, by an appeal to the first Borel-Cantelli Lemma, that the large deviations estimate \eqref{largedevsmootheq} implies that in this
case we also have almost sure convergence exactly as stated in Theorem~\ref{almostsurebasis} (b). This result has an obvious corollary.
\begin{cor}\label{largedev}
Suppose that $n_L$ is the counting measure on the zero set of the GAF $f_L$ defined via frames \eqref{framedefn} and $U$ is an open bounded subset of the
complex plane.

\noindent (a)
\[
\E\left[\frac1Ln_L(U)\right]\rightarrow\frac{1}{2\pi}\mu(U)
\]
as $L\rightarrow\infty$.

\noindent (b) Let $\delta>0$. There exists $c>0$ depending only on $\delta$, $U$ and $\mu$ such that for sufficiently large values of $L$
\[
\Pro\left[\Big|\frac{\frac1Ln_L(U)}{\frac1{2\pi}\mu(U)}-1\Big|>\delta\right]\leq e^{-cL^2}.
\]
\end{cor}
\begin{rem}
As before, the large deviations estimate combined with the first Borel-Cantelli Lemma implies that$\frac1Ln_L(U)\rightarrow\frac1{2\pi}\mu(U)$ almost
surely when $L$ is restricted to integer values.
\end{rem}

It is well known that the linear statistics are asymptotically normal for the Poisson point process, and we also show that the smooth linear statistics
for our zero sets are asymptotically normal, for large values of $L$, if the measure $\mu$ is locally flat. We shall state and prove this result only for
the GAF defined via frames (that is \eqref{framedefn}) but it is easy to verify that the proof works equally well for the GAF defined via bases
\eqref{basisdefn} since it relies only on estimates for the size of the covariance kernel.

\begin{thm}\label{asympnorm}
Let $\psi$ be a smooth function with compact support in $\C$ let $n_L$ be the counting measure on the zero set of the GAF defined via frames
\eqref{framedefn}, and suppose that measure $\mu$ is locally flat (see Definition~\ref{localflat}). Define $n(\psi,L)=\frac1L\int\psi dn_L$ as before.
Then the random variable
\[
\frac{n(\psi,L)-\E[n(\psi,L)]}{\V[n(\psi,L)]^{1/2}}
\]
converges in distribution to $\mathcal{N}(0,1)$ as $L\rightarrow\infty$.
\end{thm}

We are also interested in the `hole probability', the probability that there are no zeroes in a region of the complex plane. When we take
$\phi(z)=|z|^2/2$ then the asymptotic decay of the hole probability for the zero set of the GAF defined via bases was computed in \cite{ST3}, and the more
precise version
\[
\Pro[n_{g_L}(D(z_0,r))=0]=\exp\Big\{-\frac{e^2}{4}L^2r^4(1+o(1))\Big\}
\]
as $L\rightarrow\infty$ was obtained in \cite[Theorem 1.1]{Nis10} (we note that the author considers $L=1$ and discs of large radius centred at the
origin, however the results are equivalent by re-scaling and translation invariance). If $\phi(z)=|z|^\alpha/2$ and we consider the random function $g_L$
generated by the bases $(\frac{(L^{1/\alpha}z)^n}{c_{\alpha n}})_{n=0}^\infty$ then we see that
\[
\Pro[n_{g_L}(D(0,r))=0]=\Pro[n_{g_1}(D(0,rL^{1/\alpha}))=0].
\]
We may now use \cite[Theorem 1]{Nis11} to see that
\[
\Pro[n_{g_1}(D(0,rL^{1/\alpha}))=0]=\exp\Big\{-\frac{\alpha e^2}8r^{2\alpha}L^2(1+o(1))\Big\}
\]
as $L\rightarrow\infty$ (we omit the details), however we no longer have translation invariance.

Our first result says that we always have an upper bound of the form $e^{-cL^2}$, however we have no estimate for the lower bound in general.
\begin{thm}\label{basishole}
Suppose that $n_L$ is the counting measure on the zero set of the GAF $g_L$ defined via bases \eqref{basisdefn}. Let $U$ be a bounded open subset of the
complex plane. There exists $c>0$ depending only on $U$ and $\mu$ such that for sufficiently large values of $L$
\[
\Pro[n_L(U)=0]\leq e^{-cL^2}.
\]
\end{thm}
When we work with frames, because we have estimates for the pointwise decay of the reproducing kernel, we can prove much more. In this case we show that
we have the same upper bound (with a different constant) and that the upper bound is sharp (up to constants) under additional assumptions on the decay of
the kernel $\K_L$.

\begin{thm}\label{hole}
Suppose that $n_L$ is the counting measure on the zero set of the GAF $f_L$ defined via frames \eqref{framedefn}. Let $U$ be an open bounded subset of the
complex plane.

\noindent (a) There exist $c,C>0$ depending on $U$ and $\mu$, and $\tau\geq2$ depending only on $\mu$, such that for sufficiently large values of $L$
\[
e^{-CL^{\tau}}\leq\Pro[n_L(U)=0]\leq e^{-cL^2}.
\]

\noindent (b) If the reproducing kernel $\K_L$ has fast ${L}\sp 2$ off-diagonal decay (Definition~\ref{fastdef}) then we have $\tau=2$ in (a).
\end{thm}

\begin{rems}
1. The upper bound in this theorem follows directly from Corollary~\ref{largedev} (b)

\noindent 2. In proving this result we will give upper bounds for the value $\tau$ when we do not have fast ${L}\sp 2$ off-diagonal decay.

\noindent 3. The kernel corresponding to $\phi(z)=|z|^\alpha/2$ has fast ${L}\sp 2$ off-diagonal decay (see the Appendix).
\end{rems}

While we have stressed heretofore that our work generalises the known cases in the complex plane, we should point out that these ideas have also been
studied on manifolds. For example in \cite{SZ99}, \cite{SZ08} and \cite{SZZ} the authors study the distribution of zeros of random holomorphic sections of
(large) powers of a positive holomorphic line bundle $L$ over a compact complex manifold $M$. Theorem~\ref{almostsurebasis}, for example, is completely
analogous to \cite[Theorem 1.1]{SZ99}, although our proof is less technical. We have also used many of the ideas from \cite{SZZ} in our proof of
Theorem~\ref{largedev}, where the authors also deal with the problem of having no information about a basis. A key difference between the two settings is
the compactness of the manifold $M$, which means that the spaces of sections considered are finite dimensional with a control on the growth of the
dimension. There are also some recent results in a non-compact setting, see \cite{DMS}, however the spaces considered are still assumed to be finite
dimensional.

The paper is structured as follows: In Section~2 we give some technical results that shall be used throughout the paper, and show that the covariance
kernel for the GAFs defined via bases and frames satisfy similar size estimates. In Section~3 we prove Theorems~\ref{almostsurebasis} and \ref{variance}.
In Section~4 we show that the smooth linear statistics are asymptotically normal, under some extra regularity assumptions (Theorem~\ref{asympnorm}). In
Section~5 we prove the large deviations estimates, Theorem~\ref{largedevsmooth} and Corollary~\ref{largedev}. In Section~6 we prove
Theorem~\ref{basishole}, the upper bound for the hole probability for the zero set of the GAF defined via bases. Finally in Section~7 we compute the hole
probability for the zero set of the GAF defined via frames, Theorem~\ref{hole}.

The notation $f \lesssim g$ means that there is a constant $C$ independent of the relevant variables such that $f\le C g$, and $f\simeq g$ means that
$f\lesssim g$ and $g\lesssim f$. We frequently ignore events of probability zero.

\section{Doubling Measures and Fock Spaces}
\subsection{Technical Preliminaries}
We will always assume that $\mu$ is a doubling measure (Definition~\ref{doubling}) and that $\phi$ is a subharmonic function with $\mu=\Delta\phi$. Recall
that $\rho(z)=\rho_\mu(z)$ is the radius such that $\mu(D(z,\rho_\mu(z)))=1$. Note that all of the constants (including implicit constants) in this
section depend only on the doubling constant $C_\mu$.

\begin{lemma}{\cite[Lemma 2.1]{Ch}}\label{rho comparison}
Let $\mu$ be a doubling measure in $\mathbb C$. There exists $\gamma>0$ such that for any discs $D, D'$ of respective radius $r(D)>r(D')$ with $D\cap
D'\neq\emptyset$
\[
\left(\frac{\mu(D)}{\mu(D')}\right)^\gamma\lesssim\frac{r(D)}{r(D')} \lesssim \left(\frac{\mu(D)}{\mu(D')}\right)^{1/\gamma} .
\]
\end{lemma}
In particular, the support of $\mu$ has positive Hausdorff dimension. We will sometimes require some further regularity on the measure $\mu$, so we make
the following definition.
\begin{defn}\label{localflat}
We say that a doubling measure $\mu$ is \textit{locally flat} if given any disc $D$ of radius $r(D)$ satisfying $\mu(D)=1$ then for every disc
$D'\subseteq D$ of radius $r(D')$ we have
\[
\frac{1}{\mu(D')}\simeq\left(\frac{r(D)}{r(D')}\right)^2
\]
where the implicit constants depend only on $\mu$.
\end{defn}
Trivially $\phi(z)=|z|^2$ gives us a locally flat measure, indeed the condition $0<c<\Delta\phi<C$ ensures that the measure $\Delta\phi$ is locally flat.
Moreover there is always a regularisation of the measure $\Delta\phi$ that is locally flat (see \cite[Theorem 14]{MMO}), but we shall not pursue this line
of thought.

We have the following estimates from \cite[p. 869]{MMO}: There exist $\eta>0$, $C>0$ and $\beta\in(0,1)$ such that
\begin{equation}\label{rho is poly}
C^{-1}|z|^{-\eta}\leq\rho(z)\leq C|z|^\beta\text{ for }|z|>1
\end{equation}
and
\[
|\rho(z)-\rho(\zeta)|\leq|z-\zeta|\text{ for }z,\zeta\in\mathbb{C}.
\]
Thus $\rho$ is a Lipschitz function, and so in particular is continuous. We will write
\[
D^r(z)=D(z,r\rho(z))
\]
and
\[
D(z)=D^1(z).
\]

A simple consequence of Lemma~\ref{rho comparison} is that $\rho(z)\simeq\rho(\zeta)$ for $\zeta\in D(z)$. We shall make use of the following estimate.
\begin{lemma}[{\cite[p. 205]{Ch}}]\label{Christ}
If $\zeta\not\in D(z)$ then
\[
\frac{\rho(z)}{\rho(\zeta)}\lesssim\left(\frac{|z-\zeta|}{\rho(\zeta)}\right)^{1-t}
\]
for some $t\in(0,1)$ depending only on the doubling constant, $C_\mu$.
\end{lemma}
We will need the following estimate.
\begin{lemma}{\cite[Lemma 2.3]{Ch}}\label{Green}
There exists $C>0$ depending on $C_\mu$ such that for any $r>0$
\[
\int_{D(z,r)}\log\Big( \frac{2r}{|z-\zeta|}\Big) d\mu(\zeta)\leq C\ \mu(D(z,r))\quad\quad z\in\C.
\]
\end{lemma}

We recall that $\rho^{-2}$ can be seen as a regularisation of $\mu$. We define $d_\mu$ to be the distance induced by the metric $\rho(z)^{-2}dz\otimes
d\overline{z}$, that is
\[
d_\mu(z,\zeta)=\inf \int_0^1 |\gamma'(t)| \rho^{-1}(\gamma(t)) dt ,
\]
where the infimum is taken over all piecewise $\mathcal C^1$ curves $\gamma:[0,1]\rightarrow\C$ with $\gamma(0)=z$ and $\gamma(1)=\zeta$. We have the
following estimates:
\begin{lemma}[{\cite[Lemma 4]{MMO}}]\label{distance}
There exists $\delta>0$ such that for every $r>0$ there exists $C_r>0$ such that
\begin{itemize}
\item $\displaystyle{C_r^{-1}\frac{|z-\zeta|}{\rho(z)}\leq d_\mu(z,\zeta)\leq C_r\frac{|z-\zeta|}{\rho(z)}}\text{ if }|z-\zeta|\leq r\rho(z)$ and
\item $\displaystyle{C_r^{-1}\left(\frac{|z-\zeta|}{\rho(z)}\right)^\delta\leq d_\mu(z,\zeta)\leq C_r\left(\frac{|z-\zeta|}{\rho(z)}\right)^{2-\delta}}\text{ if }|z-\zeta|> r\rho(z)$.
\end{itemize}
\end{lemma}
%

\begin{defn} A sequence $\Lambda$ is $\rho$-\emph{separated} if there exists $\delta>0$ such that
\[
\inf_{\lambda\neq\lambda'}d_\mu(\lambda,\lambda')>\delta.
\]
\end{defn}

One consequence of Lemma~\ref{distance} is that a sequence $\Lambda$ is $\rho$-separated if and only if there exists $\delta>0$ such that
\[
|\lambda-\lambda'|\geq\delta\max(\rho(\lambda),\rho(\lambda^\prime)) \qquad \lambda\neq \lambda'.
\]
This equivalent condition is often easier to work with and is the reason we call this condition $\rho$-separated. We shall make repeated use of the
following lemma.
\begin{lemma}\label{exp epsilon}
Let $\Lambda$ be a $\rho$-separated sequence. Then for any $\epsilon>0$ and $k\ge 0$ there exists a constant $C>0$ depending only on $k$, $\epsilon$, and
$C_\mu$ such that
\begin{alist}
\item$\displaystyle{\int_\C \frac{|z-\zeta|^k}{\exp d_\mu^\epsilon(z,\zeta)} \frac{dm(z)}{\rho(z)^2}}\le C\rho^k(\zeta)$ and
\item$\displaystyle{\sum_\lil \frac{|z-\lambda|^k}{\exp d_\mu^\epsilon(z,\lambda)} }\le C\rho^k(\zeta)$.
\end{alist}
\end{lemma}

\begin{proof}
The proof of (a) is  almost identical to \cite[Lemma 2.7]{MO}. Lemma~\ref{distance} implies that there exists $\varepsilon>0$ such that
\[
\exp d_\mu^\epsilon(z,\zeta)\gtrsim\exp\left(\frac{|z-\zeta|}{\rho(\zeta )}\right)^{\varepsilon}.
\]
Let $f(x)=x^{\frac{k}{\varepsilon}}-\frac{k}{\varepsilon}x^{\frac{k}{\varepsilon}-1}$ and note that for any $y>0$
\[
\int_y^{+\infty} e^{-x}f(x)=e^{-y}y^{k/\varepsilon}.
\]
Splitting the integral over the regions $D(\zeta)$ and $\C\setminus D(\zeta)$ and using Lemma~\ref{Christ} we see that
\begin{align*}
\iC\frac{|z-\zeta|^k}{\exp d_\mu^\epsilon(z,\zeta)}\frac{dm(z)}{\rho(z)^2}&\lesssim\rho^k(\zeta)+\int_{\C\setminus
D(\zeta)}\rho^k(\zeta)\int_{\left(\frac{|z-\zeta|}{\rho(\zeta)}\right)^\varepsilon}^\infty e^{-x}f(x)dx\frac{dm(z)}{\rho(z)^2}\\
&\lesssim\rho^k(\zeta)+\rho^k(\zeta)\int_{1}^{+\infty}e^{-x}f(x)\int_{D^{x^{1/\varepsilon}}(\zeta)}\frac{dm(z)}{\rho(z)^2}dx\\
&\lesssim\rho^k(\zeta )\left(1+\int_{1}^{+\infty} e^{-x}f(x)x^{\alpha}dx\right)
\end{align*}
for some positive $\alpha$.

We may estimate the sum appearing in (b) by the integral in (a) so the result follows.
\end{proof}

We will also need the following estimates.
\begin{lemma}{\cite[Lemma 19]{MMO}}\label{bernstein}
For any $r>0$ there exists $C=C(r)>0$ such that for any $f\in \mathcal{H}(\C)$ and $z\in \C:$
\begin{alist}
\item $\displaystyle{|f(z)|^2 e^{-2\phi(z)}\le C\int_{D^r(z)}|f(\zeta)|^2 e^{-2\phi(\zeta)}\frac{dm(\zeta)}{\rho^2(\zeta)}}$.
\item $\displaystyle{|\nabla (|f|e^{-\phi})(z)|^2\le \frac{C}{\rho^2(z)}\int_{D^r(z)}|f(\zeta)|^2 e^{-2\phi(\zeta)}\frac{dm(\zeta)}{\rho^2 (\zeta)}}$.
\item If $s>r$, $\displaystyle{|f(z)|^2 e^{-2\phi(z)} \le C_{r,s} \int_{D^s(z)\setminus D^r(z)}|f(\zeta)|^2 e^{-2\phi(\zeta)}\frac{dm(\zeta)}{\rho^2(\zeta)}}$.
\end{alist}
\end{lemma}


We shall scale the measure $\mu$ by a (large) parameter $L\geq1$. We shall write $\phi_L=L\phi$, $\rho_L=\rho_{L\mu}$, $d_L=d_{L\mu}$ and
$D_L^r(z)=D(z,r\rho_L(z))$
. Note that the measures $\mu$ and $L\mu$ have the same doubling constant, so we may apply all of the results in this section to the measure $L\mu$
without changing the constants. It is clear from the definition
\[
L\mu(D(z,\rho_L(z)))=1
\]
that $\rho_L(z)<\rho(z)$ for $L>1$. Thus by Lemma~\ref{rho comparison} we have
\begin{equation}\label{rhoL}
L^\gamma\lesssim\frac{\rho(z)}{\rho_L(z)}\lesssim L^{1/\gamma}
\end{equation}
and
\[
L^{-1/\gamma}\lesssim\frac{d(z,w)}{d_L(z,w)}\lesssim L^{-\gamma}
\]
for some $\gamma\leq1$, where the implicit constants are uniform in $z$.

If the measure $\mu$ is locally flat then we see that
\begin{equation}\label{rhoLideal}
\frac{\rho(z)}{\rho_L(z)}\simeq\sqrt{L}
\end{equation}
and
\[
\frac{d(z,w)}{d_L(z,w)}\simeq\frac{1}{\sqrt{L}}.
\]

\subsection{Kernel estimates}\label{kernels}
In this section we show that the covariance kernel $K_L$ for the GAF defined via frames \eqref{framedefn} satisfies similar growth estimates to the
reproducing kernel $\K_L$, which is the covariance kernel for the GAF defined via bases \eqref{basisdefn}. We will do this by showing that it is the
reproducing kernel for a different (but equivalent) norm on the space $\FL$. We shall state and prove these results for the function $K=K_1$, but since
the constants appearing depend only on the doubling constant, they may be applied \textit{mutatis mutandis} to $K_L$.

We first note that $\K$ satisfies the following estimates.

\begin{prop}[{\cite[Lemma 21]{MMO},\cite[Theorem 1.1 and Proposition 2.11]{MO},\cite[p. 355]{CO}}]\label{kernest}
There exist positive constants $C$ and $\epsilon$ (depending only on the doubling constant for $\mu$) such that for any $z,w\in\C$
\begin{alist}
\item $|\K(z,w)|\leq C e^{\phi(z)+\phi(w)}e^{-d_\mu^\epsilon(z,w)}$,
\item $C^{-1}e^{2\phi(z)}\leq\K(z,z)\leq Ce^{2\phi(z)}$,
\item $C^{-1}/\rho(z)^2\leq\Delta\log\K(z,z)\leq C/\rho(z)^2.$
\item There exists $r>0$ such that $|\K(z,w)|\geq C e^{\phi(z)+\phi(w)}$ for all $w\in D^r(z)$.
\end{alist}
\end{prop}
\begin{rem}
The off diagonal decay estimates in Theorem~1.1  and Proposition~2.11 of \cite{MO} differ from the results just stated by factors involving $\rho$. This
is because the authors study spaces with a different norm; in \cite[Section 2.3]{MMO} it is shown that the class of spaces considered here and in
\cite{MO} is the same. However one can easily verify that minor modifications to the proof in \cite{MO} give the result just stated in the spaces we are
considering.
\end{rem}

In order to ensure that the sequence of normalised reproducing kernels $(k_\lambda)_\lil$ form a frame, we require that the sequence $\Lambda$ is
sampling. Sampling sequences in the Fock spaces we consider have been characterised in terms of a Beurling-type density. The following definitions appear
in \cite{MMO}.
\begin{defn}
A sequence $\Lambda$ is \emph{sampling} for $\Fdosa$ if there exists $C>0$ such that for every $f\in \Fdosa$
\begin{equation}\label{sampling}
C^{-1}\sum_{\lil}|f(\lambda)|^2 e^{-2\phi(\lambda)}\leq \| f\|_{\Fdosa}^2\leq C \sum_{\lil}|f(\lambda)|^2 e^{-2\phi(\lambda)}.
\end{equation}
\end{defn}

\begin{thm}[{\cite[Theorem A]{MMO}}]
A sequence $\Lambda$ is sampling for $\Fdosa$ if and only if $\Lambda$ is a finite  union of $\rho$-separated sequences containing a $\rho$-separated
subsequence $\Lambda'$  such that $\mathcal D^-_{\mu}(\Lambda')>$ where
\[
\mathcal D^-_{\mu}(\Lambda)=\liminf_{r\to \infty}\inf_{z\in\C} \frac{\# \bigl(\Lambda\cap \overline{D(z,r\rho(z))}\bigr)}
{\mu(D(z,r\rho(z)))}>\frac1{2\pi}.
\]
\end{thm}

Recall that $k_\zeta (z)=\frac{\K(z,\zeta)}{\K(\zeta,\zeta)^{1/2}}$. It is clear from Proposition~\ref{kernest} that $|\langle
k_{\lambda},f\rangle|\simeq|f(\lambda)|e^{-\phi(\lambda)}$ for all $f\in\Fdosa$. Thus $\Lambda$ is a sampling sequence for $\Fdosa$ if and only if
\[
\|f\|_{\Fdosa}^2\simeq\sum_{\lil}|\langle k_{\lambda},f \rangle|^2\quad \text{for all } f\in\Fdosa,
\]
that is, if and only if $(k_{\lambda})_\lil$ is a frame in  $\Fdosa$.

We denote the (canonical) dual frame by $(\tilde{k}_{\lambda})_\lil$, and note that any $f\in\Fdosa$ can be expanded as
\[
f=\sum_{\lil}\langle f,\tilde{k}_{\lambda}\rangle k_{\lambda}.
\]
We introduce a new inner product on the space $\Fdosa$ given by
\[
\inprod{f}{g} =\sum_{\lil}\langle f,\tilde{k}_{\lambda}\rangle\overline{\langle g,\tilde{k}_{\lambda}\rangle}
\]
and note that the norm $\renorm{f}=\inprod{f}{f}^{1/2}$ is equivalent to the original norm $\|\cdot\|_{\Fdosa}$ (if $\Lambda$ is sampling).

\begin{prop}
The reproducing kernel for the (re-normed) space $(\Fdosa,\renorm{\cdot})$ is
\[
K(z,w)=\sum_\lil k_\lambda (z)\overline{k_\lambda (w)}.
\]
\end{prop}
\begin{proof}
It is clear that, for each fixed $w\in\C$, $K(\cdot,w)=K_w$ is in the space, so we need only verify the reproducing property. Note first that $\langle
k_{\lambda'},\tilde{k}_{\lambda}\rangle=\langle \tilde{k}_{\lambda'},k_{\lambda}\rangle$. This follows from the fact that
$\tilde{k}_{\lambda}=S^{-1}k_\lambda$, where $S$ is the frame operator associated to $(k_{\lambda})_\lil$, and $S$ is self adjoint with respect to
$\langle\cdot,\cdot\rangle$. Now, for any $f\in\Fdosa$,
\begin{align*}
\inprod{f}{K_w}=\sum_{\lil}\langle f,\tilde{k}_{\lambda}\rangle\overline{\langle K_w,\tilde{k}_{\lambda}\rangle}&=\sum_{\lil}\sum_{\lpil}\langle f,\tilde{k}_{\lambda}\rangle\overline{\langle k_{\lambda'},\tilde{k}_{\lambda}\rangle}k_{\lambda'}(w)\\
&=\sum_{\lpil}k_{\lambda'}(w)\sum_{\lil}\langle f,\tilde{k}_{\lambda}\rangle\langle k_{\lambda},\tilde{k}_{\lambda'}\rangle\\
&=\sum_{\lpil}k_{\lambda'}(w)\left\langle \sum_{\lil}\langle f,\tilde{k}_{\lambda}\rangle k_{\lambda},\tilde{k}_{\lambda'}\right\rangle\\
&=\sum_{\lpil}k_{\lambda'}(w)\langle f,\tilde{k}_{\lambda'}\rangle=f(w),
\end{align*}
which completes the proof.
\end{proof}

We now show that the growth and off diagonal diagonal decay of $K$ are similar to that of $\K$.

\begin{prop}\label{covkerest}
There exist positive constants $C$, $c$ and $\epsilon$ (depending only on the doubling constant for $\mu$ and the sampling constant appearing in
\eqref{sampling}) such that for any $z,w\in\C$
\begin{alist}
\item $|K(z,w)|\leq C e^{\phi(z)+\phi(w)}e^{-cd_\mu^\epsilon(z,w)},$
\item $C^{-1}e^{2\phi(z)}\leq K(z,z)\leq Ce^{2\phi(z)}$ and
\item $C^{-1}\frac{1}{\rho(z)^2}\leq\Delta\log K(z,z)\leq C\frac{1}{\rho(z)^2}.$
\item There exists $r>0$ such that $|K(z,w)|\geq C e^{\phi(z)+\phi(w)}$ for all $w\in D^r(z)$.
\end{alist}
\end{prop}
\begin{proof}
We have (see \cite[p. 26]{Ber})
\begin{align*}
\sqrt{K(z,z)}&=\sup\{|f(z)|:f\in\Fdosa \; ,\; \renorm{f} \leq 1\}\\
&\simeq\sup\{|f(z)| : f\in \Fdosa \; ,\; \|f\|_{\Fdosa} \leq 1\}=\sqrt{\K(z,z)}
\end{align*}
and so Proposition~\ref{kernest} implies (b). Similarly (again see \cite[p. 26]{Ber})
\[
\Delta\log K(z,z)=\frac{4\sup\{|f'(z)|^2:\ f\in {\mathcal F}_\phi^2,\ f(z)=0; \renorm{f}\leq1\}}{K(z,z)}\simeq\Delta\log \K(z,z)
\]
so that (c) also follows from Proposition~\ref{kernest}.

We note that for all $w\in D^r(z)$, applying Lemma~\ref{bernstein}(b),
\[||K(w,z)|e^{-\phi(w)}-|K(z,z)|e^{-\phi(z)}|\lesssim\frac{1}{\rho(z)}\|K(\cdot,z)\|_\Fdosa|z-w|\lesssim re^{\phi(z)}\]
so that for sufficiently small $r$, (b) implies (d).

Finally we have, by the estimates in Proposition~\ref{kernest},
\[
|K(w,z)|\leq\sum_\lil |k_\lambda (z)\overline{k_\lambda (w)}|\lesssim e^{\phi(z)+\phi(w)}\sum_\lil
e^{-d_\mu^\epsilon(z,\lambda)-d_\mu^\epsilon(\lambda,w)}.
\]
Now
\[
\sum_{\lil,d_\mu(z,\lambda)>\frac{1}{2}d_\mu(z,w)}e^{-d_\mu^\epsilon(z,\lambda)-d_\mu^\epsilon(\lambda,w)}\leq
e^{-2^{-\epsilon}d_\mu^\epsilon(z,w)}\sum_{\lil}e^{-d_\mu^\epsilon(\lambda,w)}\lesssim e^{-2^{-\epsilon}d_\mu^\epsilon(z,w)}
\]
where we have used Lemma~\ref{exp epsilon}. The remaining terms satisfy $d_\mu(w,\lambda)\geq\frac{1}{2}d_\mu(z,w)$ and may be treated similarly.
\end{proof}
\begin{rems}
1. When we apply this result to $K_L$, it is important that the constants in the relation $\renorm{\cdot}_L\simeq\|\cdot\|_{\FL}$ are uniform in $L$, so
that the constant $C$ appearing in the conclusion can be taken to be uniform in $L$. It is not difficult to see that we can always do this. For each $L$
we chose a sequence $\Lambda_L$ and constants $\delta_0<R_0$ which do not depend on $L$ satisfying the following properties:
\begin{itemize}
\item The discs $(D_L^{\delta_0}(\lambda))_{\lil_L}$ are pairwise disjoint.
\item We have $\C=\cup_{\lil_L}D_L^{R_0}(\lambda)$.
\item Each $z\in\C$ is contained in at most $N_0$ discs of the form $D_L^{R_0+1}(\lambda)$ where $N_0$ does not depend on $z$ or $L$.
\end{itemize}
Applying Lemma~\ref{bernstein}(b) one can show that if $R_0$ is sufficiently small then $\sum_{\lil_L}|f(\lambda)|^2 e^{-2\phi(\lambda)}\simeq \|
f\|_{\FL}$ where the implicit constants are uniform in $L$.

2. We have used only the fact that $(k_{\lambda})_\lil$ is a frame, and the expression of the reproducing kernel as an extremal problem, to show that
$K(z,z)\simeq\K(z,z)$ and $\Delta\log K(z,z)\simeq\Delta\log\K(z,z)$. Our proof therefore carries over to any space where these properties hold.
\end{rems}

We will sometimes be able to prove sharper results if we assume some extra off-diagonal decay on the kernel $\K_L$. The condition we will use is the
following.
\begin{defn}\label{fastdef}
The kernel $\K_L$ has \emph{fast $L^2$ off-diagonal decay} if, given $C,r>0$ there exists $R>0$ (independent of $L$) such that
\begin{equation}\label{fastL2}
\sup_{z\in D^r(z_0)}e^{-2\phi_L(z)}\int_{\C\setminus D^R(z_0)}|\K_L(z,\zeta)|^2e^{-2\phi_L(\zeta)}\frac{dm(\zeta)}{\rho_L(\zeta)^2}\leq e^{-CL}
\end{equation}
for all $z_0\in\C$ and $L$ sufficiently large.
\end{defn}
\begin{rems}
1. If $\phi(z)=|z|^2/2$ then since $\K_L(z,\zeta)=e^{Lz\overline{\zeta}}/2\pi^2$ it is easy to see that the $\K_L$ has fast $L^2$ off-diagonal decay. More
generally if $\phi(z)=|z|^\alpha/2$ it can also be seen that $\K_L$ has fast $L^2$ off-diagonal decay but we postpone the proof to an appendix since it is
long and tedious.

2. We also note that \cite[Proposition 1.18]{Ch} shows that there exist $\phi$ with $0<c<\Delta\phi<C$ that do not satisfy \eqref{fastL2}, so that local
flatness does not imply fast $L^2$ decay.
\end{rems}
\section{Proof of Theorems~\ref{almostsurebasis} and \ref{variance}}
In this section we will prove Theorems~\ref{almostsurebasis} and \ref{variance}. We will follow the scheme of the proof of \cite[Theorem 1.1]{SZ99}. We
begin by proving Theorem~\ref{almostsurebasis} (a). Recall that $n_L$ is the counting measure on the zero set of the GAF defined via bases,
\eqref{basisdefn}.

\begin{proof}[Proof of Theorem~\ref{almostsurebasis}(a)]
Let $\psi$ be a smooth function with compact support in $\C$. The Edelman-Kostlan formula gives
\[
\E\left[n(\psi,L)\right]=\frac{1}{4\pi L}\iC\psi(z)\Delta\log \K_L(z,z) dm(z)
\]
so that, by Proposition~\ref{kernest}
\begin{align*}
\left|\E\left[n(\psi,L)\right]-\frac{1}{2\pi}\int\psi d\mu\right|&=\frac{1}{4\pi L}\left|\iC\Delta\psi(z)\left(\log
\K_L(z,z)-2\phi_L(z)\right)dm(z)\right|\\ &\lesssim\frac{1}{L}\iC|\Delta\psi(z)|dm(z).
\end{align*}
\end{proof}

\begin{proof}[Proof of Theorem~\ref{variance}]
We have (see \cite[Theorem~3.3]{SZ08} or \cite[Lemma~3.3]{NS})
\[
\V[n(\psi,L)]=\iC\iC\Delta\psi(z)\Delta\psi(w)J_L(z,w)dm(z)dm(w)
\]
where
\[
J_L(z,w)=\frac{1}{16\pi^2}\sum_{n=1}^\infty\frac{1}{n^2}\left(\frac{|\K_L(z,w)|^2}{\K_L(z,z)\K_L(z,w)}\right)^n\simeq\frac{|\K_L(z,w)|^2}{\K_L(z,z)\K_L(z,w)}.
\]
Fix $z\in\supp\psi$, choose $\alpha>2/\gamma$ where $\gamma$ is the constant appearing in \eqref{rhoL}, and let $\epsilon$ be the constant from
Proposition~\ref{kernest}. Write
\[
I_1=\int_{d_L(z,w)\geq(\alpha\log L)^{1/\epsilon}}\Delta\psi(w)J_L(z,w)dm(w),
\]
\[
I_2=\int_{d_L(z,w)<(\alpha\log L)^{1/\epsilon}}\left(\Delta\psi(w)-\Delta\psi(z)\right)J_L(z,w)dm(w),
\]
\[
I_3=\int_{d_L(z,w)<(\alpha\log L)^{1/\epsilon}}J_L(z,w)dm(w).
\]
and note that
\[
\iC\Delta\psi(w)J_L(z,w)dm(w)=I_1+I_2+\Delta\psi(z)I_3.
\]
Now, by Proposition~\ref{kernest}, $J_L(z,w)\lesssim e^{-d_L^\epsilon(z,w)}\leq L^{-\alpha}$ when $d_L(z,w)\geq (\alpha\log L)^{1/\epsilon}$ and so
\[
|I_1|\leq L^{-\alpha}\int_{d_L(z,w)\geq (\alpha\log L)^{1/\epsilon}}|\Delta\psi(w)|dm(w)\leq L^{-\alpha}\|\Delta\psi\|_{L^1(\C)}.
\]
Also, since $d_L(z,w)\gtrsim L^{\gamma}d_\mu(z,w)$, we see that if $z$ and $w$ satisfy $d_L(z,w)<(\alpha\log L)^{1/\epsilon}$ then
\[
\Delta\psi(w)-\Delta\psi(z)\rightarrow0\text{ as }L\rightarrow\infty,
\]
and so
\[
|I_2|\leq\sup_{d_L(z,w)<(\alpha\log L)^{1/\epsilon}}|\Delta\psi(w)-\Delta\psi(z)|I_3=o(I_3).
\]
Finally, using Proposition~\ref{kernest} and Lemma~\ref{distance}, we see that
\[
I_3\lesssim\iC e^{-d_L^\epsilon(z,w)}dm(w)\lesssim\iC e^{-c(\frac{|z-w|}{\rho_L(z)})^{\epsilon'}}dm(w)=\rho_L(z)^2\iC e^{-c'|\zeta|^{\epsilon'}}dm(\zeta).
\]
Similarly, for $r$ sufficiently small,
\[
I_3\gtrsim\int_{D_L^r(z)} dm(w)=\pi r^2\rho_L(z)^2,
\]
that is, $I_3\simeq\rho_L(z)^2$. We thus conclude that (note that $\rho_L(z)^2\gtrsim L^{-2/\gamma}\rho(z)^2\gtrsim L^{-\alpha}$)
\[
\V[n(\psi,L)]=\iC\Delta\psi(z)\left(I_1+I_2+\Delta\psi(z)I_3\right)dm(z)\simeq\iC\Delta\psi(z)^2\rho_L(z)^2dm(z)
\]
which completes the proof.
\end{proof}
We will now use the results we have just proved for the mean and the variance of the `smooth linear statistics' to prove Theorem~\ref{almostsurebasis}
(b).
\begin{proof}[Proof of Theorem~\ref{almostsurebasis} (b)]
First note that
\begin{align*}
\E\left[\Big(n(\psi,L)-\frac{1}{2\pi}\int\psi d\mu\Big)^2\right]\lesssim&\E\bigg[\Big(n(\psi,L)-\E[n(\psi,L)]\Big)^2\bigg]\\
&+\left(\E[n(\psi,L)]-\frac{1}{2\pi}\int\psi d\mu\right)^2.
\end{align*}
Now Theorem~\ref{variance} implies that
\begin{align*}
\E\Big[\big(n(\psi,L)-\E[n(\psi,L)]\big)^2\Big]&=\V[n(\psi,L)]\\
&\simeq L^{-2}\iC(\Delta\psi(z))^2\rho_L(z)^2dm(z)\\
&\lesssim L^{-2(1+\gamma)}\iC(\Delta\psi(z))^2\rho(z)^2dm(z)
\end{align*}
while (a) implies that
\[
\E[n(\psi,L)]-\frac{1}{2\pi}\int\psi d\mu=O(L^{-1}).
\]
We thus infer that
\[
\E\left[\left(n(\psi,L)-\frac{1}{2\pi}\int\psi d\mu\right)^2\right]\lesssim L^{-2}
\]
which implies that
\[
\E\left[\sum_{L=1}^\infty\left(\frac{1}{L}n(\psi,L)-\frac{1}{2\pi}\int\psi
d\mu\right)^2\right]=\sum_{L=1}^\infty\E\left[\left(\frac{1}{L}n(\psi,L)-\frac{1}{2\pi}\int\psi d\mu\right)^2\right]<+\infty.
\]
This means that
\[
\frac{1}{L}n(\psi,L)-\frac{1}{2\pi}\int\psi d\mu\rightarrow0
\]
almost surely, as claimed.
\end{proof}
\section{Asymptotic Normality}
This section consists of the proof of Theorem~\ref{asympnorm}. As we have previously noted, we shall consider only the GAF defined via frames
\eqref{framedefn}. All of the results stated here apply equally well to the GAF defined via bases \eqref{basisdefn}, and the proofs are identical except
that the estimates from Proposition~\ref{covkerest} should be replaced by the estimates from Proposition~\ref{kernest}. Our proof of
Theorem~\ref{asympnorm} is based entirely on the following result which was used to prove asymptotic normality in the case $\phi(z)=|z|^2$ (\cite[Main
Theorem]{ST1}).

\begin{thm}[{\cite[Theorem 2.2]{ST1}}]
Suppose that for each natural number $m$, $f_m$ is a Gaussian analytic function with covariance kernel $\Xi_m$ satisfying $\Xi_m(z,z)=1$ and let $n_m$ be
counting measure on the set of zeroes of $f_m$. Let $\nu$ be a measure on $\C$ satisfying $\nu(\C)=1$ and suppose $\Theta:\C\to\R$ is a bounded measurable
function. Define $Z_m=\iC \log(|f_m(z)|)\Theta(z)d\nu(z)$ and suppose that
\begin{equation}\label{normratio}
\liminf_{m\rightarrow\infty}\frac{\iint_{\C^2}|\Xi_m(z,w)|^{2}\Theta(z)\Theta(w)d\nu(z)d\nu(w)}{\sup_{z\in\C}\iC|\Xi_m(z,w)|^2d\nu(z)} >0,
\end{equation}
and that
\begin{equation}\label{norm0}
\lim_{m\rightarrow\infty}\sup_{z\in\C}\iC|\Xi_m(z,w)|d\nu(w) = 0.
\end{equation}
Then the distributions of the random variables
\[
\frac{Z_m-\E Z_m}{\sqrt{\V Z_m}}
\]
converge weakly to $\mathcal{N}(0,1)$ as $m\rightarrow\infty$.
\end{thm}

\begin{rem}
In fact in \cite{ST1} the authors prove a more general result, but we shall only require the form we have stated. We have also slightly modified the
denominator in condition \eqref{normratio}, but it is easy to verify that this does not affect the proof (cf. \cite[Section 2.5]{ST1}).
\end{rem}

\begin{proof}[Proof of Theorem~\ref{asympnorm}]
We consider instead the random variable $n(\psi,L)=\int\psi dn_L$ since it is clear that the factor $L^{-1}$ is unimportant. We first note that Green's
formula implies that
\[
n(\psi,L)=\frac{1}{2\pi}\iC\Delta\psi(z)\log|f_L(z)|dm(z)
\]
which combined with the Edelman-Kostlan formula gives
\[
Z_L(\psi)=n(\psi,L)-\E[n(\psi,L)]=\frac{1}{2\pi}\iC\Delta\psi(z)\log\frac{|f_L(z)|}{K_L(z,z)^{1/2}}dm(z).
\]
Write $\hat{f}_L(z)=\frac{f_L(z)}{K_L(z,z)^{1/2}}$, $\Theta(z)=\frac{c}{2\pi}\Delta\psi(z)\rho(z)^2$ and
$d\nu(z)=\frac{1}{c}\chi_{\supp\psi}(z)\frac{dm(z)}{\rho(z)^2}$ where the constant $c$ is chosen so that $\nu(\C)=1$. Note that
\[
Z_L(\psi)=\iC\log|\hat{f}_L(z)|\Theta(z)d\nu(z)
\]
so we need only check that conditions \eqref{normratio} and \eqref{norm0} hold to show asymptotic normality. Here
$\Xi_L(z,w)=\frac{K_L(z,w)}{K_L(z,z)^{1/2}K_L(w,w)^{1/2}}$. Now, by the estimates of Proposition~\ref{covkerest} and \eqref{rhoLideal},
\begin{align*}
\iC|\Xi_L(z,w)|d\nu(w)&\simeq e^{-\phi_L(z)}\iC|K_L(z,w)|e^{-\phi_L(w)}d\nu(w)\\
&\leq e^{-\phi_L(z)}\left(\iC|K_L(z,w)|^2e^{-2\phi_L(w)}\frac{dm(w)}{\rho(w)^2}\right)^{1/2}\nu(\C)^{1/2}\\
&\overset{(*)}{\simeq} L^{-\frac{1}{2}}e^{-\phi_L(z)}\left(\iC|K_L(z,w)|^2e^{-2\phi_L(w)}\frac{dm(w)}{\rho_L(w)^2}\right)^{1/2}\simeq L^{-\frac{1}{2}},
\end{align*}
where we have used local flatness \eqref{rhoLideal} for the estimate $(*)$, so \eqref{norm0} holds. (In fact to prove \eqref{norm0} it suffices to use the
estimate \eqref{rhoL}.) Similarly
\[
\iC|\Xi_L(z,w)|^2d\nu(w)\simeq L^{-1}.
\]
By a computation almost identical to that in the proof of Theorem~\ref{variance} we also have
\begin{align*}
\iint_{\C^2}|\Xi_L(z,w)|^{2}\Theta(z)\Theta(w)d\nu(z)d\nu(w)&\simeq\iC(\Delta\psi(z))^2\rho_L(z)^2dm(z)\\
&\simeq L^{-1}\iC(\Delta\psi(z))^2\rho(z)^2dm(z)
\end{align*}
which verifies \eqref{normratio}. (In both of these estimates we use \eqref{rhoLideal} since the estimate \eqref{rhoL} is not enough, it is here that our
local flatness assumption is important.)
\end{proof}
\section{Large deviations}
In this section we prove Theorem~\ref{largedevsmooth} and Corollary~\ref{largedev}. We borrow many of the ideas used here from \cite{ST3} and \cite{SZZ},
but some modifications are necessary to deal with the fact that $\phi$ is non-radial and we are in a non-compact setting. The key ingredient in the proof
of Theorem~\ref{largedevsmooth} is the following lemma.
\begin{lemma}\label{keyupper}
For any disc $D=D^r(z_0)$ and any $\delta>0$ there exists $c>0$ depending only on $\delta$, $D$ and $\mu$ such that
\[
\int_{D}\left|\log|f_L(z)|-\phi_L(z)\right|dm(z)\leq\delta L
\]
outside an exceptional set of probability at most $e^{-cL^2}$, for $L$ sufficiently large.
\end{lemma}
We begin with the following lemma.
\begin{lemma}\label{maxlogphi}
Given a disc $D=D^r(z_0)$ and $\delta>0$ there exists $c>0$ depending only on the doubling constant such that
\[
\left|\max_{z\in \overline{D}}\big(\log|f_L(z)|-\phi_L(z)\big)\right|\leq\delta L
\]
outside an exceptional set of probability at most $e^{-c\delta\mu(D)L^2}$, for $L$ sufficiently large.
\end{lemma}
\begin{proof}
Define $\hat{f}_L(z)=\frac{f_L(z)}{K_L(z,z)^{1/2}}$. We will show that
\[
\Pro\left[\left|\max_{z\in \overline{D}}\log|\hat{f}_L(z)|\right|\geq\delta L\right]\leq e^{-c\delta\mu(D)L^2}
\]
for $L$ sufficiently large, which will imply the claimed result by Proposition~\ref{covkerest}(b). We divide the proof in two parts.

1. We first show that
\[
\Pro\left[\max_{z\in \overline{D}}|\hat{f}_L(z)|\leq e^{-\delta L}\right]\leq e^{-c\delta\mu(D)L^2}.
\]
For each $L$ define $S_L$ to be a $\rho_L$-separated sequence with the constant
\[
R=\inf\{d_L(s,t)\colon s\neq t\text{ and }s,t\in S_L\}
\]
to be chosen (large but uniform in $L$). Moreover we assume that
\[
\sup_{z\in\C}d_L(z,S_L)<\infty,
\]
uniformly in $L$ once more. Trivially
\[
\Pro\left[\max_{z\in \overline{D}}|\hat{f}_L(z)|\leq e^{-\delta L}\right]\leq\Pro\left[|\hat{f}_L(s)|\leq e^{-\delta L}\text{ for all }
s\in\overline{D}\cap S_L\right]
\]
and we now estimate the probability of this event. We write
\[
\overline{D}\cap S_L=\{s_1,\ldots,s_N\}.
\]
Note that for $R_1$ sufficiently small
\[
L\mu(D^{2r}(z_0))\leq\sum_{j=1}^NL\mu(D_L^{R_1}(s_j))\lesssim N
\]
while for $R_2$ large enough
\[
L\mu(D^{r}(z_0))\geq\sum_{j=1}^NL\mu(D_L^{R_2}(s_j))\gtrsim N
\]
so that $N\simeq L\mu(D)$. Consider the vector
\[
\xi=\begin{pmatrix} \hat{f}_L(s_1)\\
\vdots\\
\hat{f}_L(s_N)
\end{pmatrix}
\]
which is a mean-zero $N$-dimensional complex normal with covariance matrix $\sigma$ given by
\[
\sigma_{mn}=\frac{K_L(s_m,s_n)}{K_L(s_m,s_m)^{1/2}K_L(s_n,s_n)^{1/2}}.
\]
Note that $\sigma_{nn}=1$ and $|\sigma_{mn}|\lesssim e^{-d_L^\epsilon(s_n,s_m)}$ so that if the sequence $S_L$ is chosen to be sufficiently separated then
the components of the vector $\xi$ will be `almost independent'. We write $\sigma=I+A$ and note that
\begin{align*}
\max_n\left|\sum_{m\neq n}\sigma_{mn}\right|\lesssim\max_n\sum_{m\neq n}e^{-d_L^\epsilon(s_n,s_m)}&\lesssim\max_n\int_{\C\setminus B_L(s_n,R)}e^{-d_L^\epsilon(s_n,w)}\frac{dm(w)}{\rho_L(w)^2}\\
&\lesssim\int_{R^{\epsilon'}}^\infty x^\alpha e^{-x}dx
\end{align*}
for some $\alpha,\epsilon'>0$ by an argument identical to that given in the proof of Lemma~\ref{exp epsilon}. Thus by choosing $R$ sufficiently large we
have $\|A\|_\infty\leq\frac{1}{2}$ and so for any $v\in\C^N$
\[
\|\sigma v\|_\infty\geq\frac{1}{2}\|v\|_\infty.
\]
Thus the eigenvalues of $\sigma$ are bounded below by $\frac1 2$ and so if $\sigma=BB^*$ then
\[
\|B^{-1}\|_2\leq\sqrt{2}.
\]
Now the components of the vector $\zeta=B^{-1}\xi$ are iid $\mathcal{N}_\C(0,1)$ random variables, which we denote $\zeta_j$, and moreover
\[
\|\zeta\|_\infty\leq\|B^{-1}\xi\|_2\leq\sqrt{2}\|\xi\|_2\leq\sqrt{2N}\|\xi\|_\infty.
\]
This means that
\begin{align*}
\Pro\left[|\hat{f}_L(s)|\leq e^{-\delta L}\text{ for all }s\in \overline{D}\cap S_L\right]&\leq\Pro\left[|\zeta_j|\leq\sqrt{2N}e^{-\delta L}\text{ for all
}1\leq
j\leq N\right]\\
&=\left(1-\exp(-2Ne^{-2\delta L})\right)^N\leq e^{-c\delta\mu(D)L^2}
\end{align*}
for $L$ sufficiently large, where c depends only on the doubling constant (and $\sup_{z\in\C}d_L(z,S_L)$), as claimed.

2. For the second part of the proof we must estimate
\[
\Pro[\max_{z\in\overline{D}}|\hat{f}_L(z)|\geq e^{\delta L}]
\]
and so we define the event
\[
\mathcal E =\{\max_{z\in\overline{D}}|\hat{f}_L(z)|\geq e^{\delta L}\}.
\]
We write $\tilde{\Lambda}_L=\Lambda_L\cap D^{2r}(z_0)$ and $\tilde{f}_L=\sum_{\lambda\in\tilde{\Lambda}_L}a_\lambda k_\lambda (z)$, and note that
$\#\tilde{\Lambda}_L\simeq L\mu(D)$ as in the first part of the proof. Consider the event
\[
\mathcal A =\{|a_\lambda|\leq L\frac{|\lambda-z_0|}{\rho_L(z_0)}\text{ for }\lambda\in\Lambda_L\setminus\tilde{\Lambda}_L\}.
\]
If the event $\mathcal A$ occurs and $z\in\overline{D}$ then, since $d_L(\lambda,z)\geq C_r d_L(\lambda,z_0)$ for some $C_r>0$, we have by
Proposition~\ref{covkerest}(a)
\begin{align*}
\left|\frac{f_L(z)-\tilde{f}_L(z)}{K_L(z,z)^{1/2}}\right|&\lesssim\sum_{\lambda\in\Lambda_L\setminus\tilde{\Lambda}_L}|a_\lambda|e^{-d_L^\epsilon(\lambda,z)}\\
&\leq \frac{L}{\rho_L(z_0)}\sum_{\lambda\in\Lambda_L\setminus\tilde{\Lambda}_L}|\lambda-z_0|e^{-C_r^\epsilon d_L^\epsilon(\lambda,z)}\\
&\lesssim\frac{L}{\rho_L(z_0)}\int_{\C\setminus D^r(z_0)}|\zeta-z_0|e^{-C_r^\epsilon d_L^\epsilon(\zeta,z)}\frac{dm(\zeta)}{\rho_L(\zeta)^2}\lesssim L
\end{align*}
where the final estimate comes from an argument similar to that used in the proof of Lemma~\ref{exp epsilon} and the implicit constant depends on $r$.
Hence the event $\mathcal A\cap\mathcal E$ implies that
\[
\max_{z\in\overline{D}}\left|\frac{\tilde{f}_L(z)}{K_L(z,z)^{1/2}}\right|\geq e^{\delta L}-C'_r L\geq e^{\frac{\delta L}{2}}
\]
for $L$ sufficiently large, where $C'_r$ is another positive constant. Now a simple application of the Cauchy-Schwartz inequality shows that
\[
|\tilde{f}_L(z)|\leq\left(\sum_{\lambda\in\tilde{\Lambda}_L}|a_\lambda|^2\right)^{1/2}\left(\sum_{\lambda\in\tilde{\Lambda}_L}|k_\lambda(z)|^2\right)^{1/2}\leq\left(\sum_{\lambda\in\tilde{\Lambda}_L}|a_\lambda|^2\right)^{1/2}K_L(z,z)^{1/2}
\]
and so
\begin{align*}
\Pro[\mathcal A\cap\mathcal E]&\leq\Pro\left[\max_{z\in\overline{D}}\left|\frac{\tilde{f}_L(z)}{K_L(z,z)^{1/2}}\right|\geq e^{\frac{\delta L}{2}}\right]\\
&\leq\Pro\left[\sum_{\lambda\in\tilde{\Lambda}_L}|a_\lambda|^2\geq e^{\delta L}\right]\\
&\leq\Pro\left[|a_\lambda|^2\geq \frac{e^{\delta L}}{\#\tilde{\Lambda}_L}\text{ for all }\lambda\in\tilde{\Lambda}_L\right] =\left(\exp-\frac{e^{\delta
L}}{\#\tilde{\Lambda}_L}\right)^{\#\tilde{\Lambda}_L}=e^{-e^{\delta L}}.
\end{align*}
We finally estimate the probability of the event $\mathcal A$; using \eqref{rhoL} and \eqref{rho is poly} we see that
\begin{align}\label{tailprob}
\log\Pro[\mathcal A]&=\log\prod_{\lil_L\setminus\tilde{\Lambda}_L}\left(1-\exp\left(-L^2\frac{|\lambda-z_0|^2}{\rho_L^2(z_0)}\right)\right)\notag\\
&\simeq-\sum_{\lil_L\setminus\tilde{\Lambda}_L}\exp\left(-L^2\frac{|\lambda-z_0|^2}{\rho_L^2(z_0)}\right)\notag\\
&\gtrsim-\int_{\C\setminus D}\exp\left(-L^2\frac{|\zeta-z_0|^2}{\rho_L^2(z_0)}\right)\frac{dm(\zeta)}{\rho_L(\zeta)^2}\notag\\
&\gtrsim-L^{2/\gamma}\int_{\C\setminus D}\exp\left(-CL^{2+2/\gamma}\frac{|\zeta-z_0|^2}{\rho^2(z_0)}\right)\frac{dm(\zeta)}{\rho(\zeta)^2}\notag\\
&\gtrsim -C_0 L^{2/\gamma} e^{-C_1 L^{2+2/\gamma}}
\end{align}
where $C_0$ and $C_1$ depend on $r$ and the doubling constant, and the final estimate uses an argument similar to that given in the proof of
Lemma~\ref{exp epsilon}. We finally compute that
\[
\Pro[\mathcal E]\leq\Pro[\mathcal E\cap\mathcal A]+\Pro[\mathcal A^{c}]\leq e^{-e^{\delta L}}+(1-\exp\{-C_0 L^{2/\gamma} e^{-C_1 L^{2+2/\gamma}}\})\leq
e^{-cL^2}
\]
for $L$ sufficiently large and for any positive $c$.
\end{proof}
\begin{lemma}\label{keydisc}
Given a disc $D=D^r(z_0)$ there exist $c,C>0$ depending only on the doubling constant such that
\[
\int_{D}\left|\log|f_L(z)|-\phi_L(z)\right|dm(z)\leq C r^2\rho(z_0)^2\mu(D) L
\]
outside of an exceptional set of probability at most $e^{-c\mu(D)L^2}$, for $L$ sufficiently large.
\end{lemma}
We will use the following result to prove this lemma.
\begin{thm}[{\cite[Chap. 1 Lemma 7]{Pas} or \cite[Theorem 1]{AC}}]\label{Pascuas}
If $u$ is a subharmonic function on $\overline{\mathbb{D}}$ then, for all $\zeta\in\mathbb{D}$,
\[
u(\zeta)=\int_{\mathbb{D}}\widetilde{P}(\zeta,z)u(z)dm(z)-\int_{\mathbb{D}}\widetilde{G}(\zeta,z)\Delta u(z)
\]
where
\[
\widetilde{P}(\zeta,z)=\frac1\pi\frac{(1-|\zeta|^2)^2}{|1-\overline{z}\zeta|^4}
\]
and
\[
\widetilde{G}(\zeta,z)=\frac1{4\pi}\left(\log\Big|\frac{1-\overline{\zeta}z}{\zeta-z}\Big|^2-\Big(1-\Big|\frac{\zeta-z}{1-\overline{\zeta}z}\Big|^2\Big)\right)
\]
\end{thm}
\begin{proof}[Proof of Lemma~\ref{keydisc}]
Applying Lemma~\ref{maxlogphi} we see that outside of an exceptional set, we may find $\zeta\in D^{r/2}(z_0)$ such that
\[
-L\mu(D)\leq\log|f_L(\zeta)|-\phi_L(\zeta).
\]
Making the appropriate change of variables in Theorem~\ref{Pascuas} and applying the resulting decomposition to the subharmonic functions $\log|f_L|$ and
$\phi_L$ on $D$ we see that
\begin{align*}
\log|f_L(\zeta)|-\phi_L(\zeta)&=\int_{D}\widetilde{P}\left(\frac{\zeta-z_0}{r\rho(z_0)},\frac{z-z_0}{r\rho(z_0)}\right)(\log|f_L(z)|-\phi_L(z))\frac{dm(z)}{r^2\rho(z_0)^2}\\
&\qquad-\int_{D}\widetilde{G}\left(\frac{\zeta-z_0}{r\rho(z_0)},\frac{z-z_0}{r\rho(z_0)}\right)(2\pi dn_L(z)-\Delta\phi_L(z))\\
&\leq\int_{D}\widetilde{P}\left(\frac{\zeta-z_0}{r\rho(z_0)},\frac{z-z_0}{r\rho(z_0)}\right)(\log|f_L(z)|-\phi_L(z))\frac{dm(z)}{r^2\rho(z_0)^2}\\
&\qquad+\int_{D}\widetilde{G}\left(\frac{\zeta-z_0}{r\rho(z_0)},\frac{z-z_0}{r\rho(z_0)}\right)\Delta\phi_L(z)
\end{align*}
since $\widetilde{G}$ is always positive. Now, since $\zeta\in D^{r/2}(z_0)$, we have by Lemma~\ref{Green}
\begin{align*}
\int_{D}\widetilde{G}\left(\frac{\zeta-z_0}{r\rho(z_0)},\frac{z-z_0}{r\rho(z_0)}\right)\Delta\phi_L(z)&\leq\int_{D^{r/2}(\zeta)}\widetilde{G}\left(\frac{\zeta-z_0}{r\rho(z_0)},\frac{z-z_0}{r\rho(z_0)}\right)\Delta\phi_L(z)\\
&\qquad+\int_{D\setminus D^{r/2}(\zeta)}\widetilde{G}\left(\frac{\zeta-z_0}{r\rho(z_0)},\frac{z-z_0}{r\rho(z_0)}\right)\Delta\phi_L(z)\\
&\lesssim L\int_{D^{r/2}(\zeta)}\log\left(\frac{\frac32r}{|\zeta-z|}\right)d\mu(z)+L\mu(D)\\
&\lesssim L\mu(D)
\end{align*}
and so
\[
0\leq\int_{D}\widetilde{P}\left(\frac{\zeta-z_0}{r\rho(z_0)},\frac{w-z_0}{r\rho(z_0)}\right)(\log|f_L(w)|-\phi_L(w))\frac{dm(w)}{r^2\rho(z_0)^2}+CL\mu(D)
\]
for some positive $C$ depending only on the doubling constant. Noting that $\widetilde{P}$ is also positive and satisfies
\[
\widetilde{P}\left(\frac{\zeta-z_0}{r\rho(z_0)},\frac{w-z_0}{r\rho(z_0)}\right)\simeq1
\]
for $w\in D$ and $\zeta\in D^{r/2}(z_0)$ we see that
\[
\int_{D}\log^{-}(|f_L(w)|e^{-\phi_L(w)})\frac{dm(w)}{r^2\rho(z_0)^2}\lesssim\int_{D}\log^{+}(|f_L(w)|e^{-\phi_L(w)})\frac{dm(w)}{r^2\rho(z_0)^2}+L\mu(D)
\]
and so
\[
\int_{D}|\log|f_L(w)|-\phi_L(w)|\frac{dm(w)}{r^2\rho(z_0)^2}\lesssim\int_{D}\log^{+}(|f_L(w)|e^{-\phi_L(w)})\frac{dm(w)}{r^2\rho(z_0)^2}+L\mu(D).
\]
Applying Lemma~\ref{maxlogphi} once more we see that outside of another exceptional set
\[
\int_{D}\log^{+}(|f_L(w)|e^{-\phi_L(w)})\frac{dm(w)}{r^2\rho(z_0)^2}\lesssim L\mu(D)
\]
which completes the proof.
\end{proof}
We are now ready to prove Lemma~\ref{keyupper}
\begin{proof}[Proof of Lemma~\ref{keyupper}]
Given $\delta>0$ we may cover $D$ with discs $(D^{r_j}(z_j))_{j=1}^N$ such that $\mu(D^{r_j}(z_j))=\delta$ and $z_j\in U$. The Vitali covering lemma
implies that we may assume that $N\lesssim \mu(D)/\delta$. Now, applying Lemma~\ref{keydisc} we see that outside of an exceptional set
\[
\int_{U}\left|\log|f_L(z)|-\phi_L(z)\right|dm(z)\leq\delta L\sum_{j=1}^N r_j^2\rho(z_j)^2.
\]
We finally note that $\rho(z_j)\simeq\rho(z_0)$ and that Lemma~\ref{Christ} implies that
\[
r_j\lesssim\delta^{\gamma}
\]
for all $j$. Thus
\[
\int_{U}\left|\log|f_L(z)|-\phi_L(z)\right|dm(z)\lesssim\delta LN\delta^{\gamma}\lesssim L\delta^{2\gamma}.
\]
Appropriately changing the value of $\delta$ completes the proof.
\end{proof}
\begin{proof}[Proof of Theorem~\ref{largedevsmooth}]
We have already noted that the proof of (a) is identical to the proof of Theorem~\ref{almostsurebasis} (a). It remains to show the large deviations
estimate (b). We first note that
\begin{align*}
\Big|n(\psi,L)-\frac{1}{2\pi}\int\psi d\mu\Big|&=\frac{1}{2\pi L}\left|\iC\Delta\psi(z)(\log|f_L(z)|-\phi_L(z))dm(z)\right|\\
&\leq\frac{1}{2\pi L}\max_{z\in\C}|\Delta\psi(z)|\int_{\supp\psi}|\log|f_L|-\phi_L|dm
\end{align*}
and so applying Lemma~\ref{keyupper} with $\delta'=\delta|\int\psi d\mu|/\|\Delta\psi\|_\infty$ we see that
\[
\left|n(\psi,L)-\frac{1}{2\pi}\int\psi d\mu\right|\leq\delta \left|\frac{1}{2\pi}\int\psi d\mu\right|
\]
outside an exceptional set of probability at most $e^{-cL^2}$, as claimed.
\end{proof}

\begin{proof}[Proof of Corollary~\ref{largedev}]
Let $\delta>0$ and choose smooth, compactly supported $\psi_1$ and $\psi_2$ satisfying
\[
0\leq\psi_1\leq\chi_U\leq\psi_2\leq1,
\]
\[
\iC\psi_1d\mu\geq\mu(U)(1-\delta)
\]
and
\[
\iC\psi_2d\mu\leq\mu(U)(1+\delta).
\]
(a) Applying Theorem~\ref{largedevsmooth} (a) we see that, for $L$ sufficiently large,
\begin{align*}
\E\left[\frac1Ln_L(U)\right]-\frac{1}{2\pi}\mu(U)&\leq\E\left[\frac1L\int\psi_2dn_L\right]-\frac{1}{2\pi}\mu(U)\\
&\leq\frac{1}{2\pi}\int\psi d\mu+\frac{C}{L}\iC|\Delta\psi_2(z)|dm(z)-\frac{1}{2\pi}\mu(U)\\
&\leq\frac{\delta}{2\pi}\mu(U)+\frac{C}{L}\iC|\Delta\psi_2(z)|dm(z).
\end{align*}
Similarly
\[
\E\left[\frac1Ln_L(U)\right]-\frac{1}{2\pi}\mu(U)\geq-\frac{\delta}{2\pi}\mu(U)-\frac{C}{L}\iC|\Delta\psi_2(z)|dm(z).
\]
Choosing first $\delta$ small and then $L$ large (depending on $\delta$) completes the proof of (a).

\noindent (b) Outside an exceptional set of probability $e^{-cL^2}$ we have, by Theorem~\ref{largedevsmooth} (b)
\[
\frac1Ln(\psi_2,L)\leq(1+\delta)\frac1{2\pi}\iC\psi_2d\mu.
\]
We see that
\[
\frac1Ln_L(U)\leq\frac1Ln(\psi_2,L)\leq(1+\delta)\frac1{2\pi}\iC\psi_2d\mu\leq(1+\delta)^2\frac{\mu(U)}{2\pi}
\]
whence
\[
\frac{\frac1Ln_L(U)}{\frac{\mu(U)}{2\pi}}-1\lesssim\delta.
\]
Similarly
\[
\frac{\frac1Ln_L(U)}{\frac{\mu(U)}{2\pi}}-1\gtrsim-\delta.
\]
outside another exceptional set of probability $e^{-cL^2}$, which after appropriately changing the value of $\delta$ completes the proof.
\end{proof}
\section{Proof of Theorem~\ref{basishole}}
We will use some of the ideas from the proof of Theorem~\ref{largedev} here. We begin with a lemma that is very similar to Lemma~\ref{maxlogphi}. It is
clear that if we could prove an exact analogue of Lemma~\ref{maxlogphi} then we could prove a large deviations theorem, since it is only in the proof of
this lemma that we use the decay estimates for the frame elements. Unfortunately we are unable to prove such a result, but the following result will be
enough to prove a hole theorem. Recall that we write $D=D^r(z_0)$.
\begin{lemma}\label{basismaxlog}
Given $z_0\in\C$ and $\delta,r>0$ there exists $c>0$ depending only on the doubling constant such that
\[
\Pro\left[\max_{z\in\overline{D}}\big(\log|g_L(z)|-\phi_L(z)\big)\leq-\delta L\right]\leq e^{-c\delta\mu(D)L^2}
\]
for $L$ sufficiently large. Moreover there exists $C'>0$ depending on $\phi$ and $r$ such that for all $z_0\in\C$ and $C>C'$
\[
\Pro\left[\max_{z\in\overline{D}}\big(\log|g_L(z)|-\phi_L(z)\big)\geq CL\right]\leq e^{-cL^2}
\]
for sufficiently large $L$.
\end{lemma}
\begin{proof}
The proof that
\[
\Pro\left[\max_{z\in\overline{D}}\big(\log|g_L(z)|-\phi_L(z)\big)\leq-\delta L\right]\leq e^{-c\delta\mu(D)L^2}
\]
is identical to the proof of the first part of Lemma~\ref{maxlogphi}, we omit the details.

To prove the second estimate we use the following result, which is simply \cite[Lemma 2.4.4]{HKPV} translated and re-scaled.
\begin{lemma}
Let $f$ be a Gaussian analytic function in a neighbourhood of the disc $D(z_0,R)$ with covariance kernel $K$. Then for $r<R/2$ we have
\[
\Pro\left[\max_{z\in\overline{D(z_0,r)}}|f(z)|>t\right]\leq 2e^{-t^2/8\sigma_{2r}^2}
\]
where $\sigma_{2r}^2=\max\{K(z,z):z\in\overline{D(z_0,2r)}\}$.
\end{lemma}
Let $C_1=\min\{\phi(z):z\in\overline{D}\}$ and $C_2=\max\{\phi(z):z\in\overline{D^{2r}(z_0)}\}$. Note that
\[
\max\{\K_L(z,z):z\in D(z_0,2r)\}\lesssim e^{2C_2L}.
\]
Hence
\begin{align*}
\Pro\left[\max_{z\in\overline{D}}\big(\log|g_L(z)|-\phi_L(z)\big)\geq CL\right]&\leq \Pro\left[\max_{z\in\overline{D}}|g_L(z)|\gtrsim e^{(C+C_1)L}\right]\\
&\leq 2\exp\{-c'e^{2(C+C_1-C_2)L}\}\leq e^{-cL^2}
\end{align*}
for any $c>0$ if $C+C_1-C_2>0$.
\end{proof}
We may now immediately infer the following lemma. All integrals over circles are understood to be with respect to normalised Lebesgue measure on the
circle.
\begin{lemma}\label{basiskey}
For any $z_0\in\C$ and any $\delta,r>0$ there exists $c>0$ depending only on $\delta$, $\Delta\phi(D)$ and the doubling constant such that
\[
\Pro\left[\int_{\partial D}\big(\log|g_L(z)|-\phi_L(z)\big)\leq-\delta L\right]\leq e^{-cL^2}
\]
for $L$ sufficiently large.
\end{lemma}
\begin{proof}
It suffices to show this for small $\delta$. Put $\kappa=1-\delta^{1/4}$, $N=[2\pi\delta^{-1}]$ and define $z_j=z_0+\kappa r \rho(z_0)\exp(\frac{2\pi
ij}{N})$ and $D_j=D(z_j,\delta r\rho(z_0))$ for $j=1,\ldots,N$. Lemma~\ref{basismaxlog} implies that outside an exceptional set of probability at most
$Ne^{-c\delta\mu(D_j)L^2}\leq e^{-c'L^2}$ (where $c'$ depends on $\delta$, $\mu(D^r(z_0))$ and the doubling constant) there exist
$\zeta_j\in\overline{D_j}$ such that
\[
\log|g_L(\zeta_j)|-\phi_L(\zeta_j)\geq-\delta L.
\]
Let $P(\zeta,z)$ and $G(\zeta,z)$ be, respectively, the Poisson kernel and the Green function for $D$ where we use the convention that the Green function
is positive. Applying the Riesz decomposition to the subharmonic functions $\log|g_L|$ and $\phi_L$ on the disc $D$ implies that
\begin{align*}
-\delta L&\leq\frac1 N\sum_{j=1}^N\big(\log|g_L(\zeta_j)|-\phi_L(\zeta_j)\big)\\
&=\int_{\partial D}\big(\log|g_L(z)|-\phi_L(z)\big)+\int_{\partial D}\Big(\frac1
N\sum_{j=1}^NP(\zeta_j,z)-1\Big)\big(\log|g_L(z)|-\phi_L(z)\big)\\
&\qquad-\int_D\frac1 N\sum_{j=1}^NG(\zeta_j,z)\big(dn_L(z)-\Delta\phi_L(z)\big)\\
&\leq\int_{\partial D}\big(\log|g_L(z)|-\phi_L(z)\big)+\int_{\partial D}\Big(\frac1
N\sum_{j=1}^NP(\zeta_j,z)-1\Big)\big(\log|g_L(z)|-\phi_L(z)\big)\\
&\qquad+\int_D\frac1 N\sum_{j=1}^NG(\zeta_j,z)\Delta\phi_L(z).
\end{align*}
\begin{claim}\label{claim0}
There exists $\widetilde{C}>0$ such that
\[
\int_{\partial D}\big|\log|g_L(z)|-\phi_L(z)\big|\leq \widetilde{C}\mu(D)L
\]
outside of an exceptional set of probability at most $e^{-cL^2}$.
\end{claim}
\begin{claim}[{\cite[Claim 2]{ST3}}]\label{claim1}
There exists $C_0>0$ such that
\[
\max_{z\in D}\left|\frac1 N\sum_{j=1}^NP(\zeta_j,z)-1\right|\leq C_0\delta^{1/2}
\]
\end{claim}
\begin{claim}\label{claim2}
There exists $C_1>0$ and $0<\alpha<1/4$ depending only on the doubling constant and $\mu(D)$ such that
\[
\int_D G(\zeta_j,z)\Delta\phi_L(z)\leq C_1\delta^{\alpha}L
\]
for $\delta$ sufficiently small.
\end{claim}
Applying Claims~\ref{claim0} and \ref{claim1} we see that outside another exceptional set we have
\[
\left|\int_{\partial D}\Big(\frac1 N\sum_{j=1}^NP(\zeta_j,z)-1\Big)\big(\log|g_L(z)|-\phi_L(z)\big)\right|\lesssim\delta^{1/2}L
\]
while Claim~\ref{claim2} implies that
\[
\int_D\frac1 N\sum_{j=1}^NG(\zeta_j,z)\Delta\phi_L(z)\leq C_1\delta^{\alpha}L.
\]
Hence
\[
\int_{\partial D}\left(\log|g_L(z)|-\phi_L(z)\right)\geq-(\delta+C_0\delta^{3/2}+C_1\delta^{\alpha})L\gtrsim-\delta^{\alpha}L
\]
outside an exceptional set, and so the lemma follows.
\end{proof}
\begin{proof}[Proof of Claim~\ref{claim0}]
We use the same notation. Lemma~\ref{basismaxlog} implies that outside an exceptional set of probability at most $e^{-cL^2}$ there exists
$\zeta_0\in\overline{D^{r/2}(z_0)}$ such that
\[
\log|g_L(\zeta_j)|-\phi_L(\zeta_j)\geq-\mu(D) L.
\]
Another application of the Riesz decomposition to the subharmonic functions $\log|g_L|$ and $\phi_L$ on the disc $D$ implies that
\begin{align*}
-\mu(D)L&\leq\log|g_L(\zeta_0)|-\phi_L(\zeta_0)\\
&=\int_{\partial D}P(\zeta_0,z)\big(\log|g_L(z)|-\phi_L(z)\big)-\int_DG(\zeta_0,z)\big(2\pi dn_L(z)-\Delta\phi_L(z)\big)\\
&\leq\int_{\partial D}P(\zeta_0,z)\big(\log|g_L(z)|-\phi_L(z)\big)+L\int_DG(\zeta_0,z)\Delta\phi(z).
\end{align*}
Now since $\zeta_0\in D^{r/2}(z_0)$, we have by Lemma~\ref{Green}
\begin{align*}
\int_DG(\zeta_0,z)\Delta\phi(z)&\leq\int_{D^{r/2}(\zeta_0)}G(\zeta_0,z)\Delta\phi(z)+\int_{D\setminus
D^{r/2}(w)}G(\zeta_0,z)\Delta\phi(z)\\
&\lesssim\int_{D^{r/2}(\zeta_0)}\log\left(\frac{\frac32r}{|\zeta_-z|}\right)d\mu(z)+\mu(D)\\
&\lesssim \mu(D)
\end{align*}
and so
\[
0\leq\int_{\partial D}P(\zeta_0,z)\big(\log|g_L(z)|-\phi_L(z)\big)+CL\mu(D)
\]
for some positive $C$ depending only on the doubling constant. Thus
\[
\int_{\partial D}P(\zeta_0,z)\log^{-}(|g_L(z)|e^{-\phi_L(z)})\leq\int_{\partial D}P(\zeta_0,z)\log^{+}(|g_L(z)|e^{-\phi_L(z)})+CL\mu(D).
\]
We note that for $z\in\partial D$ and $\zeta_0\in D^{r/2}(z_0)$ we have
\[
\frac13\leq P(\zeta_0,z)\leq3
\]
and so
\[
\int_{\partial D}|\log|g_L(z)|-\phi_L(z)|\lesssim\int_{\partial D}\log^{+}(|g_L(z)|e^{-\phi_L(z)})+L\mu(D).
\]
Applying Lemma~\ref{maxlogphi} once more we see that outside of another exceptional set
\[
\int_{\partial D}\log^{+}(|g_L(z)|e^{-\phi_L(z)})\lesssim L\mu(D)
\]
which completes the proof.
\end{proof}
\begin{proof}[Proof of Claim~\ref{claim2}]
To simplify the notation we move to the unit disc. We write $\varphi(w)=\phi(z_0+r\phi(z_0)w)$ for $w\in\D$ and $w_j=(\zeta_j-z_0)/r\rho(z_0)$ and note
that $1-|w_j|\lesssim\delta^{1/4}$. We see that
\[
\int_D G(\zeta_j,z)\Delta\phi_L(z)=\frac L{2\pi}\int_\D\log\left|\frac{1-\overline{w_j}w}{w-w_j}\right|\Delta\varphi(w)
\]
and we write
\[
B_j(r)=\{w\in\D:\left|\frac{w-w_j}{1-\overline{w_j}w}\right|\leq r\}=D\Big(\frac{1-r^2}{1-r^2|w_j|^2}w_j,\frac{1-|w_j|^2}{1-r^2|w_j|^2}r\Big)
\]
for the hyperbolic discs of centre $w_j$ and radius $r$. Fix some $\beta<1/4$ and note that
\[
\int_{\D\setminus B_j(1-\delta^{\beta})}\log\left|\frac{1-\overline{w_j}w}{w-w_j}\right|\Delta\varphi(w)
\leq-\log(1-\delta^{\beta})\Delta\varphi(\D)\leq2\delta^{\beta}\mu(D).
\]
Also, using the distribution function, we see that
\begin{align*}
\int_{B_j(1-\delta^{\beta})}\log\left|\frac{1-\overline{w_j}w}{w-w_j}\right|\Delta\varphi(w)&=\int_0^\infty\Delta\varphi(B_j(1-\delta^{\beta})\cap
B_j(e^{-x}))dx\\
&=\int_0^{-\log(1-\delta^{\beta})}\Delta\varphi(B_j(1-\delta^{\beta}))dx\\
&\qquad+\int_{-\log(1-\delta^{\beta})}^\infty\Delta\varphi(B_j(e^{-x}))dx\\
&\leq2\delta^{\beta}\mu(D)+\int_{-\log(1-\delta^{\beta})}^\infty\Delta\varphi(B_j(e^{-x}))dx.
\end{align*}
Now the Euclidean radius of the disc $B_j(e^{-x})$ is given by
\[
\frac{1-|w_j|^2}{1-e^{-2x}|w_j|^2}e^{-x}\lesssim\frac{1-|w_j|}{1-e^{-x}|w_j|}\lesssim\frac{\delta^{1/4}}{\delta^{\beta}}
\]
which gets arbitrarily small, while $\rho_{\Delta\varphi}(w)\simeq\rho_{\Delta\varphi}(0)$ for all $w\in\D$. Applying Lemma~\ref{Christ} to the doubling
measure $\Delta\varphi$ we see that there exists $0<\gamma<1$ such that
\begin{align*}
\int_{-\log(1-\delta^{\beta})}^\infty\Delta\varphi(B_j(e^{-x}))dx&
\lesssim\int_{\delta^\beta}^\infty\left(\frac{1-|w_j|^2}{1-e^{-2x}|w_j|^2}e^{-x}\right)^\gamma dx\\
&\lesssim(1-|w_j|)^\gamma\int_{\delta^\beta}^\infty\frac{e^{-\gamma x}}{(1-e^{-x}|w_j|)^\gamma}dx\\
&\leq(1-|w_j|)^\gamma|w_j|^{-\gamma}\int_0^1\frac{(1-u)^{\gamma-1}}{u^\gamma}du\leq C_\gamma\delta^{\gamma/4}.
\end{align*}
where we have made the change of variables $u=1-e^{-x}|w_j|$. We therefore have
\[
\int_D G(\zeta_j,z)\Delta\phi_L(z)\lesssim(\delta^{\gamma/4}+\delta^{\beta})L
\]
and the claim follows by choosing $\alpha=\min\{\gamma/4,\beta\}$.
\end{proof}

We are now ready to prove Theorem~\ref{basishole}. Since we do not have any control on the dependence of the constants on the bounded set $U$, we assume
that $U$ is the disc $D$.
\begin{proof}[Proof of Theorem~\ref{basishole}]
Suppose that $g_L$ has no zeroes in $D$. Recall that we use $G(\zeta,z)$ to denote the Green function for $D$. Applying Jensen's formula to $g_L$ and the
Riesz decomposition to the subharmonic function $\phi_L$ on the disc $D$ we see that
\[
\log|g_L(z_0)|-\phi_L(z_0)=\int_{\partial D}\left(\log|g_L(z)|-\phi_L(z)\right)+L\int_DG(z_0,z)\Delta\phi(z)dm(z).
\]
Choosing $\delta=\int_DG(z_0,z)\Delta\phi(z)/2$ in Lemma~\ref{basiskey} shows that outside an exceptional set of probability at most $e^{-cL^2}$ we have
\[
\log|g_L(\zeta_0)|-\phi_L(\zeta_0)\geq\delta L.
\]
Now Proposition~\ref{kernest} shows that
\[
\Pro[\log|g_L(\zeta_0)|-\phi_L(\zeta_0)\geq\delta L]\leq\Pro\Big[\frac{|g_L(\zeta_0)|}{\K_L(z_0,z_0)^{1/2}}\gtrsim e^{\delta L}\Big]\leq\exp\{-Ce^{2\delta
L}\}\leq e^{-cL^2}
\]
and so
\[
\Pro[n_L(D)=0]\leq e^{-cL^2},
\]
which completes the proof.
\end{proof}
\section{Proof of Theorem~\ref{hole}}
We have previously remarked that the upper bound in Theorem~\ref{hole} is a simple consequence of Theorem~\ref{largedev}, we now prove the lower bounds.
We will do this by first finding a deterministic function $h_L$ that does not vanish in the hole and then constructing an event that ensures the GAF $f_L$
is `close' to $h_L$. Since we can always find a disc $D=D^r(z_0)$ contained in $U$, and we do not have any control on the dependence of the constants on
$U$, we will prove the theorem only in the case $U=D$. We begin by constructing the function $h_L$.

\begin{lemma}\label{nonzero}
There exists an entire function $h_L$ with the following properties:
\begin{itemize}
\item $\|h_L\|_{\FL}=1$.
\item There exists $C_0>0$ depending on $\mu(D)$ and the doubling constant such that
\[
|h_L(z)|e^{-\phi_L(z)}\geq e^{-C_0 L}
\]
for all $z\in D$.
\end{itemize}
\end{lemma}
\begin{rem}
In the case $\phi(z)=|z|^2$ we may take $h_L$ to be constant. More generally, if
\[
\iC e^{-2\phi_L}\frac{dm}{\rho_L^2}\leq C^L
\]
then we can take $h_L=C^{-L}$. In general, however, it may not even be the case that
\[
\iC e^{-2\phi}\frac{dm}{\rho^2}
\]
is finite (consider $\phi(z)=(\re z)^2$).
\end{rem}
\begin{proof}
Let $\K_\delta(z,w)$ be the reproducing kernel for the space $\mathcal F_{\delta\phi}^2$ and consider the normalised reproducing kernel
\[
k_\delta(z)=\frac{\K_\delta(z,z_0)}{\K_\delta(z_0,z_0)^{1/2}}.
\]
Now since $\rho_{\delta\mu}(z_0)\rightarrow\infty$ as $\delta\rightarrow0$ Proposition~\ref{kernest} shows that there exists $\delta_0$ and $C>0$
(depending only on $r$ and the doubling constant) such that
\begin{equation}\label{kdeltalower}
|k_\delta(z)|e^{-\delta\phi(z)}\geq C
\end{equation}
for all $z\in D$ and all $\delta<\delta_0$. Given any $L$ sufficiently large we can find $\delta\in[\delta_0/2,\delta_0]$ and an integer $N$ such that
$L=N\delta$. We note that $\rho_{\delta\mu}(z)\simeq\rho_\mu(z)$ for all $\delta$ in this range (where the implicit constants depend on $\delta_0$) and so
applying Proposition~\ref{kernest} and \eqref{rhoL} gives
\begin{align*}
\iC|k_\delta(z)|^{2N}e^{-2\phi_L(z)}\frac{dm(z)}{\rho_L(z)^2}&\lesssim L^{2/\gamma}\iC(|k_\delta(z)|e^{-\delta\phi(z)})^{2N}\frac{dm(z)}{\rho_{\delta\mu}(z)^2}\\
&\lesssim L^{2/\gamma}\iC e^{-d_{\delta\phi}^\epsilon(z,z_0)}\frac{dm(z)}{\rho_{\delta\mu}(z)^2}\lesssim L^{2/\gamma}.
\end{align*}
Hence $k_\delta^N$ is an entire function in $\Fdosa_L$ and we define $h_L=k_\delta^N/\|k_\delta^N\|_{\Fdosa_L}$. We finally note that \eqref{kdeltalower}
implies that for all $z\in D$
\[
|h_L(z)|e^{-\phi_L(z)}=(|k_\delta(z)|e^{-\delta\phi(z)})^N/\|k_\delta^N\|_{\Fdosa_L}\gtrsim C^N L^{1/\gamma}\geq e^{-C_0 L}
\]
where $C_0$ depends on $\delta_0$ and the doubling constant.
\end{proof}

\begin{proof}[Proof of the upper bounds in Theorem~\ref{hole}]
(a) Recall that $(\tilde{k}_{\lambda})_{\lil_L}$ is the dual frame associated to the frame $(k_{\lambda})_{\lil_L}$. Since $h_L\in\mathcal F_{\phi_L}^2$
we may write $h_L=\sum_{\lil}\langle h_L,\tilde{k}_{\lambda}\rangle k_{\lambda}=\sum_{\lil}c_\lambda k_{\lambda}$ where we define $c_\lambda=\langle
h_L,\tilde{k}_{\lambda}\rangle$ (and we ignore the dependence on $L$ to simplify the notation). Note that, for any $z\in D$, we have
\[
|f_L(z)|e^{-\phi_L(z)}=\Big|h_L(z)+\sum_{\lil}(a_\lambda-c_\lambda)\tilde{k}_{\lambda}(z)\Big|e^{-\phi_L(z)}\geq e^{-C_0 L}
-\sum_{\lil}|a_\lambda-c_\lambda||\tilde{k}_{\lambda}(z)|e^{-\phi_L(z)}.
\]
We therefore have
\[
\Pro[n_L(D)=0]\geq\Pro[\max_{z\in D}\sum_{\lil}|a_\lambda-c_\lambda||\tilde{k}_{\lambda}(z)|e^{-\phi_L(z)}<e^{-C_0 L}]
\]
and we now estimate the probability of this event. First define
\[
\alpha=\max\{0,\frac1\delta(\frac1{\epsilon}-\gamma)\}
\]
where $\epsilon,\gamma$ and $\delta$ are the constants appearing in Proposition~\ref{kernest}, \eqref{rhoL} and Lemma~\ref{distance} respectively. Fix a
large positive constant $C_1$ to be specified, write
\[
D_L=D^{C_1L^\alpha r}(z_0)
\]
and define the event
\[
\mathcal{E}_1=\{|a_\lambda-c_\lambda|\leq L\frac{|\lambda-z_0|}{\rho_L(z_0)}:\lambda\in\Lambda_L\setminus D_L\}.
\]
If $\mathcal{E}_1$ occurs and $z\in D$ then, using an argument identical to that given in the proof of Lemma~\ref{exp epsilon}, we see that
\begin{align*}
\sum_{\lambda\in\Lambda_L\setminus
D_L}|a_\lambda-c_\lambda||k_{\lambda}(z)|e^{-\phi_L(z)}&\lesssim L\sum_{\lambda\in\Lambda_L\setminus D_L}\frac{|\lambda-z_0|}{\rho_L(z_0)}e^{-d_L^{\epsilon}(z,\lambda)}\\
&\lesssim L^{1+1/\gamma}\sum_{\lambda\in\Lambda_L\setminus D_L}\frac{|\lambda-z_0|}{\rho(z_0)}e^{-c'L^{\epsilon\gamma} d_\phi^{\epsilon}(z_0,\lambda)}\\
&\lesssim L^{1+1/\gamma}\int_{\C\setminus D_L}\frac{|\zeta-z_0|}{\rho(z_0)}e^{-c'L^{\epsilon\gamma}\left(\frac{|\zeta-z_0|}{\rho(z_0)}\right)^{\epsilon\delta}}\frac{dm(z)}{\rho_L(\zeta)^2}\\
&\lesssim L^{\beta_0}\int_{c'C_1^{\delta\epsilon}L^{\alpha'}}^{+\infty} e^{-t}t^{\beta_1}dt\\
&\leq \frac{1}{2}e^{-C_0 L}
\end{align*}
for $C_1$ sufficiently large, where $\alpha'=\max\{1,\epsilon\gamma\}$, and $\beta_0$ and $\beta_1>0$ are some exponents that depend on the doubling
constant.

We define the event
\[
\mathcal{E}_2=\{|a_\lambda-c_\lambda|\leq\frac{e^{-C_0 L}}{C_2\sqrt{\#\Lambda_L\cap D_L}}:\lambda\in\Lambda_L\cap D_L\},
\]
where $C_2$ is a positive constant to be chosen. Note that for all $z\in\C$, $\mathcal{E}_2$ implies that by choosing $C_2$ sufficiently large
\begin{align*}
\sum_{\lambda\in\Lambda_L\cap D_L}|a_\lambda-c_\lambda||k_{\lambda}(z)|e^{-\phi_L(z)}&\leq\Big(\sum_{\lambda\in\Lambda_L\cap D_L}|a_\lambda-c_\lambda|^2\Big)^{1/2} \Big(\sum_{\lambda\in\Lambda_L\cap D_L}|k_{\lambda}(z)|^2\Big)^{1/2}e^{-\phi_L(z)}\\
&\leq\frac{e^{-C_0 L}}{C_2}K_L(z,z)^{1/2}e^{-\phi_L(z)}<\frac1 2e^{-C_0 L}.
\end{align*}

Hence
\[
\Pro[n_L(D)=0]\geq\Pro[\mathcal{E}_1]\Pro[\mathcal{E}_2]
\]
Recalling the definition of the coefficients $c_\lambda$ we note that
\[
\sum_{\lil_L}|c_\lambda|^2\simeq\|h_L\|_{\Fdosa_L}^2=1
\]
and so the coefficients $c_\lambda$ are bounded. This means that
\[
\Pro\Big[|a_\lambda-c_\lambda|\leq L\frac{|\lambda-z_0|}{\rho_L(z_0)}\Big]\geq\Pro\Big[|a_\lambda|\leq L\frac{|\lambda-z_0|}{2\rho_L(z_0)}\Big]
\]
when $\lambda\in\Lambda_L\setminus D_L$ and $L$ is large. We may estimate $\Pro[\mathcal{E}_1]$ similarly to \eqref{tailprob} in the proof of
Lemma~\ref{maxlogphi}. This yields $\Pro[\mathcal{E}_1]\geq1/2$ for large $L$.

Finally since $\#\Lambda_L\cap D_L\simeq\Delta\phi_L(D_L)\lesssim L^{1+\alpha/\gamma}$ we have
\[
\Pro[\mathcal{E}_2]=\prod_{\lil_L\cap D_L}\Pro\Big[|a_\lambda-c_\lambda|\leq\frac{e^{-C_0 L}}{C_2\sqrt{\#\Lambda_L\cap D_L}}\Big]\geq\Big(C\frac{e^{-2C_0
L}}{\#\Lambda_L\cap D_L}\Big)^{\#\Lambda_L\cap D_L}\geq e^{-cL^{2+\alpha/\gamma}}
\]
for some positive constants $C$ and $c$. Considering the two possible values of $\alpha$ completes the proof of the lower bounds in Theorem~\ref{hole}.

\noindent (b) We assume that the reproducing kernel $\K_L$ satisfies the estimate \eqref{fastL2}. We will use the same notation as before. Let $C_3$ and
$C_4$ be constants to be chosen and define the following events
\[
\mathcal{A}_1=\{|a_\lambda-c_\lambda|\leq L\frac{|\lambda-z_0|}{\rho_L(z_0)}:\lambda\in\Lambda_L\setminus D^{C_3r}(z_0)\}
\]
\[
\mathcal{A}_2=\{|c_\lambda-a_\lambda|\leq\frac{e^{-C_0 L}}{C_4\sqrt{\#\Lambda_L\cap D^{C_3r}(z_0)}}:\lambda\in\Lambda_L\cap D^{C_3r}(z_0)\}.
\]
We have already seen that the event $\mathcal{A}_1$ implies that
\[
\Big|\sum_{\lambda\in\Lambda_L\setminus D_L}(a_\lambda-c_\lambda)\tilde{k}_{\lambda}(z)\Big|e^{-\phi_L(z)}\leq \frac{1}{2}e^{-C_0 L}
\]
for $z\in D$. We write $\tilde{\Lambda}_L=\Lambda_L\cap(D_L\setminus D^{C_3r}(z_0))$. Note that $\mathcal{A}_1$ and \eqref{fastL2} imply that, for $z\in
D$,
\begin{align*}
&\sum_{\lambda\in\tilde{\Lambda}_L}|a_\lambda-c_\lambda||k_{\lambda}(z)|e^{-\phi_L(z)}\\
&\qquad\leq\Big(\sum_{\lambda\in\tilde{\Lambda}_L}|a_\lambda-c_\lambda|^2\Big)^{1/2}
\Big(\sum_{\lambda\in\tilde{\Lambda}_L}|k_{\lambda}(z)|^2\Big)^{1/2}e^{-\phi_L(z)}\\
&\qquad\lesssim L^{1+\alpha+1/\gamma}\sqrt{\#\tilde{\Lambda}_L}\left(\int_{\C\setminus D^{C_3r}(z_0))}|\K_L(z,\zeta)|e^{-2\phi_L(\zeta)}\frac{dm(z)}{\rho_L(z)^2}\right)^{1/2}e^{-\phi_L(z)}\\
&\qquad<\frac1 4e^{-C_0 L}
\end{align*}
for an appropriately large choice of $C_3$ and for all large $L$. By an identical computation to before we see that $\mathcal{A}_2$ implies that for all
$z\in\C$
\[
\sum_{\lambda\in\Lambda_L\cap D^{C_3r}(z_0)}|a_\lambda-c_\lambda||k_{\lambda}(z)|e^{-\phi_L(z)}<\frac1 4e^{-C_0 L}
\]
by choosing $C_4$ sufficiently large. It remains only to estimate the probabilities of the events $\mathcal{A}_1$ and $\mathcal{A}_2$, which are again
identical to the previous computation. This completes the proof.
\end{proof}

\section*{Appendix: The case $|z|^\alpha/2$}\label{radial}
We consider the space $\FL$ when $\phi(z)=|z|^\alpha/2$ and we first note that for $|z|\leq1$
\[
\rho(z)\simeq1
\]
and that
\[
\rho(z)\simeq|z|^{1-\alpha/2}
\]
otherwise. We begin by showing that the set $(\frac{(L^{1/\alpha}z)^n}{c_{\alpha n}})_{n=0}^\infty$ is an orthonormal basis for some choice of $c_{\alpha
n}\simeq\Gamma(\frac2\alpha n+1)^{1/2}$. It is clear that the functions $z^n$ are orthogonal because $\phi_L$ (and therefore $\rho_L$) are radial, and so
we need only compute the appropriate normalising constants
\[
\|z^n\|_{\FL}^2=\iC|z|^{2n}e^{-L|z|^\alpha}\frac{dm(z)}{\rho_L(z)^2}.
\]
Now it is easy to see that for $|z|\leq\rho_L(0)$
\[
\rho_L(z)\simeq\rho_L(0)\simeq L^{-1/\alpha}
\]
and that
\[
\rho_L(z)\simeq L^{-1/2}|z|^{1-\alpha/2}
\]
otherwise. Hence, using the fact that $L\rho_L(0)^{\alpha}\simeq1$, we have
\begin{align*}
\|z^n\|_{\FL}^2&=\iC|z|^{2n}e^{-L|z|^\alpha}\frac{dm(z)}{\rho_L(z)^2}\\
&\simeq L^{2/\alpha}\int_{D^{L}(0)}|z|^{2n}e^{-L|z|^\alpha}dm(z)+L\int_{\C\setminus D^{L}(0)}|z|^{2n}e^{-L|z|^\alpha}|z|^{\alpha-2}dm(z)\\
&\simeq L^{-2n/\alpha}(\int_0^{L\rho_L(0)^{\alpha}}u^{1+(2n+2)/\alpha}e^{-u}du+\int_{L\rho_L(0)^{\alpha}}^{\infty}u^{2n/\alpha}e^{-u}du)\\
&\simeq L^{-2n/\alpha}\Gamma\left(\frac2\alpha n+1\right)^{1/2}.
\end{align*}
It follows that, for some coefficients $c_{\alpha n}\simeq\Gamma(\frac2\alpha n+1)^{1/2}$, the set $(\frac{(L^{1/\alpha}z)^n}{c_{\alpha n}})_{n=0}^\infty$
is an orthonormal basis for $\FL$ and the reproducing kernel for this space is then given by
\[
\K_L(z,w)=\sum_{n=0}^\infty\frac{(L^{2/\alpha}z\overline{w})^n}{c_{\alpha n}^2}.
\]
We recall that for positive $a$ the Mittag-Leffler function
\[
E_{a,1}(\zeta)=\sum_{n=0}^\infty\frac{\zeta^n}{\Gamma(an+1)}
\]
is an entire function of order $1/a$ satisfying
\[
E_{a,1}(x)\lesssim e^{x^{1/a}}
\]
for all real positive $x$.

We now show that $\K_L$ has fast $L^2$ off-diagonal decay, that is, given $C,r>0$ there exists $R>0$ (independent of $L$) such that
\[
\sup_{z\in D^r(z_0)}e^{-L|z|^{\alpha}}\int_{\C\setminus D^{2R}(z_0)}|\K_L(z,w)|^2e^{-L|w|^{\alpha}}\frac{dm(w)}{\rho_L(w)^2}\leq e^{-CL}
\]
for all $z_0\in\C$ and $L$ sufficiently large (we have replaced $R$ by $2R$ to simplify the notation in what follows). Choosing $R$ sufficiently large we
have
\[
\int_{\C\setminus D^{2R}(z_0)}|\K_L(z,w)|^2e^{-L|w|^{\alpha}}\frac{dm(w)}{\rho_L(w)^2}\leq\int_{\C\setminus
D(0,R)}|\K_L(z,w)|^2e^{-L|w|^{\alpha}}\frac{dm(w)}{\rho_L(w)^2}
\]
and we note again that $\phi_L$ and $\rho_L$ are radial. Thus, for any positive integers $n$ and $m$,
\begin{align*}
\int_{\C\setminus D(0,R)}w^n\overline{w}^me^{-L|w|^{\alpha}}\frac{dm(w)}{\rho_L(w)^2}&\simeq\delta_{nm}\int_{\C\setminus
D(0,R)}|w|^{2n}e^{-L|w|^{\alpha}}L|w|^{\alpha-2}dm(w)\\
&=2\pi\delta_{nm}\int_{R}^{\infty}r^{2n}e^{-Lr^{\alpha}}Lr^{\alpha-1}dr\\
&=\frac{2\pi}{\alpha}\delta_{nm}L^{-2n/\alpha}\int_{LR^{\alpha}}^{\infty}u^{2n/\alpha}e^{-u}du\\
&=\frac{2\pi}{\alpha}\delta_{nm}L^{-2n/\alpha}\Gamma\left(\frac2\alpha n+1,LR^{\alpha}\right)
\end{align*}
where $\Gamma(a,z)=\int_{z}^{\infty}t^{a-1}e^{-t}dt$ denotes the incomplete Gamma function.  Now, recalling the expression for the kernel $\K_L$ we see
that
\[
|\K_L(z,w)|^2=\sum_{m,n=0}^{\infty}\frac{L^{2(m+n)/\alpha}}{c_{\alpha m}c_{\alpha n}}z^m\overline{z}^n\overline{w}^mw^n
\]
and so
\begin{align*}
\int_{\C\setminus D(0,R)}|\K_L(z,w)|^2e^{-L|w|^{\alpha}}\frac{dm(w)}{\rho_L(w)^2}&=\sum_{m,n=0}^{\infty}\frac{L^{2(m+n)/\alpha}}{c_{\alpha m}c_{\alpha
n}}z^m\overline{z}^n\int_{\C\setminus D(0,R)}w^n\overline{w}^me^{-L|w|^{\alpha}}\frac{dm(w)}{\rho_L(w)^2}\\
&\simeq\sum_{n=0}^{\infty}\frac{L^{4n/\alpha}}{\Gamma(\frac2\alpha n+1)^2}|z|^{2n}L^{-2n/\alpha}\Gamma\left(\frac2\alpha n+1,LR^{\alpha}\right)\\
&=\sum_{n=0}^{\infty}\frac{(L^{2/\alpha}|z|^{2})^n}{\Gamma(\frac2\alpha n+1)}\frac{\Gamma(\frac2\alpha n+1,LR^{\alpha})}{\Gamma(\frac2\alpha n+1)}.
\end{align*}
We split this sum in two parts. Choose $N=[\frac\alpha4LR^\alpha]$ and note that for $n\leq N$ we have, by standard estimates for the incomplete Gamma
function,
\[
\Gamma\left(\frac2\alpha n+1,LR^{\alpha}\right)\simeq(LR^{\alpha})^{2n/\alpha}e^{-LR^{\alpha}}
\]
as $R\rightarrow\infty$. Now Stirling's approximation shows that
\begin{align*}
\frac{(LR^\alpha)^{\frac2\alpha n}}{\Gamma(\frac2\alpha n+1)}&\leq\frac{(LR^\alpha)^{\frac2\alpha N}}{\Gamma(\frac2\alpha N+1)}\\
&\simeq\left(\frac4\alpha N\right)^{\frac2\alpha N}\left(\frac2\alpha N+1\right)^{1/2}\left(\frac e{\frac2\alpha N+1}\right)^{\frac2\alpha N+1}\\
&\lesssim\left(\frac{\frac4\alpha N}{\frac2\alpha N+1}\right)^{\frac2\alpha N}e^{\frac2\alpha N}\\
&\lesssim 2^{\frac2\alpha N}e^{\frac2\alpha N}\\
&=e^{LR^\alpha(1+\log2)/2}
\end{align*}
and so
\[
\frac{\Gamma(\frac2\alpha n+1,LR^{\alpha})}{\Gamma(\frac2\alpha n+1)}\lesssim e^{-cLR^{\alpha}}.
\]
It follows that
\begin{align*}
\sum_{n=0}^{N}\frac{(L^{2/\alpha}|z|^{2})^n}{\Gamma(\frac2\alpha n+1)}\frac{\Gamma(\frac2\alpha n+1,LR^{\alpha})}{\Gamma(\frac2\alpha n+1)}&\lesssim e^{-cLR^{\alpha}}\sum_{n=0}^{\infty}\frac{(L^{2/\alpha}|z|^{2})^n}{\Gamma(\frac2\alpha n+1)}\\
&=e^{-cLR^\alpha}E_{\frac2\alpha,1}(L^{2/\alpha}|z|^{2})\\
&\lesssim e^{-cLR^\alpha}e^{L|z|^\alpha}.
\end{align*}
To deal with the remaining terms we first note that
\[
\frac{\Gamma(\frac2\alpha n+1,LR^{\alpha})}{\Gamma(\frac2\alpha n+1)}\leq1
\]
for all $n$. We now choose $R$ so large that
\[
L^{2/\alpha}|z|^2<e^{-4/\alpha}L^{2/\alpha}R^2<e^{-4/\alpha}\left(\frac2\alpha N+1\right)^{2/\alpha}
\]
for $z\in D^r(z_0)$. Note that another application of Stirling's approximation yields, for any $n>N$,
\[
\Gamma\left(\frac2\alpha n+1\right)\gtrsim\Gamma\left(\frac2\alpha N+1\right)\left(\frac2\alpha N+1\right)^{2(n-N)/\alpha}.
\]
We conclude that for $z\in D^r(z_0)$ and $R$ sufficiently large we have
\begin{align*}
\sum_{n>N}\frac{(L^{2/\alpha}|z|^{2})^n}{\Gamma(\frac2\alpha n+1)}\frac{\Gamma(\frac2\alpha n+1,LR^{\alpha})}{\Gamma(\frac2\alpha
n+1)}&\lesssim\frac{(L^{2/\alpha}|z|^{2})^N}{\Gamma(\frac2\alpha N+1)}\sum_{n=0}^{\infty}\frac{(L^{2/\alpha}|z|^{2})^n}{(\frac2\alpha
N+1)^{2n/\alpha}}\\
&=\frac{(L^{2/\alpha}|z|^{2})^N}{\Gamma(\frac2\alpha N+1)}\left(1-\frac{L^{2/\alpha}|z|^{2}}{(\frac2\alpha N+1)^{2/\alpha}}\right)^{-1}\\
&\simeq\frac{(L^{2/\alpha}|z|^{2})^N}{\Gamma(\frac2\alpha N+1)}.
\end{align*}
A final appeal to Stirling's approximation yields
\begin{align*}
\frac{(L^{2/\alpha}|z|^{2})^N}{\Gamma(\frac2\alpha N+1)}&\simeq\left(L^{2/\alpha}|z|^{2}\right)^N\left(\frac2\alpha N+1\right)^{1/2}\left(\frac
e{\frac2\alpha N+1}\right)^{\frac2\alpha N+1}\\
&\simeq\left(\frac{L^{2/\alpha}|z|^{2}}{(\frac2\alpha N+1)^{2/\alpha}}\right)^N\left(\frac2\alpha N+1\right)^{-1/2}e^{2N/\alpha}\\
&\lesssim e^{-4N/\alpha}e^{2N/\alpha}\\
&=e^{-LR^\alpha/2}.
\end{align*}
Retracing our footsteps we see that we have shown that
\[
\int_{\C\setminus D^{2R}(z_0)}|\K_L(z,w)|^2e^{-L|w|^{\alpha}}\frac{dm(w)}{\rho_L(w)^2}\lesssim e^{-cLR^{\alpha}}(1+e^{L|z|^\alpha})
\]
for all $z\in D^r(z_0)$ and $R$ sufficiently large. Hence
\[
\sup_{z\in D^r(z_0)}e^{-L|z|^{\alpha}}\int_{\C\setminus D^{2R}(z_0)}|\K_L(z,w)|^2e^{-L|w|^{\alpha}}\frac{dm(w)}{\rho_L(w)^2}\leq e^{-CL}
\]
for an appropriately large $R$, as claimed.

\end{document}